\documentclass[11pt, reqno, a4paper]{amsart}
\usepackage[english]{babel}
\usepackage{amsmath,a4wide}
\usepackage{mathabx}
\usepackage{xcolor}
\usepackage{pgfplots}
\usepackage[normalem]{ulem}
\usepackage{float}
\usepackage{dsfont}
\usepackage[pdftex,
pdftitle={Euler--Maruyama approximations of the stochastic heat equation on the sphere},
pdfauthor={A. Lang and I. Motschan-Armen},
bookmarksopen,
colorlinks,
linkcolor=black,
urlcolor=black,
citecolor=black
]{hyperref}

\usepackage{enumitem}

\newtheorem{theorem}{Theorem}[section]
\newtheorem{lemma}[theorem]{Lemma}
\newtheorem{proposition}[theorem]{Proposition}

\theoremstyle{remark}

\theoremstyle{definition}

\newcommand{\R}{{\mathbb R}}
\newcommand{\C}{{\mathbb C}}
\newcommand{\N}{{\mathbb N}}
\newcommand{\E}{{\mathbb E}}

\providecommand{\norm}[1]{\left\lVert#1\right\rVert}

\newcommand{\dd}{\mathrm{d}}

\newcommand{\IS}{\mathbb{S}}
\newcommand{\IT}{\mathbb{T}}
\newcommand{\IP}{\mathbb{P}}

\newcommand{\cA}{\mathcal{A}}
\newcommand{\cF}{\mathcal{F}}
\newcommand{\cY}{\mathcal{Y}}

\newcommand{\ga}{\alpha}
\newcommand{\gb}{\beta}
\newcommand{\gd}{\delta}

\newcommand{\gk}{\kappa}

\newcommand{\gO}{\Omega}

\let\Re\relax
\DeclareMathOperator{\Re}{Re}
\let\Im\relax
\DeclareMathOperator{\Im}{Im}

\DeclareMathOperator{\trace}{Tr}

\usepackage{mathtools}

\definecolor{dkgreen}{rgb}{0,0.6,0}
\definecolor{cfblue}{RGB}{100,149,237}
\definecolor{bviolet}{RGB}{138,43,226}
\definecolor{mauve}{rgb}{0.58,0,0.82}

\usepackage{color}
\usepackage{caption}
\usepackage[labelformat=simple]{subcaption}
\usepackage{graphicx}
\usepackage[rightcaption]{sidecap}

\title[EM approximations of the stochastic heat equation on the sphere]{Euler--Maruyama approximations of the stochastic heat equation on the sphere}

\usepackage{amsaddr}
                       
\author[A.~Lang and I.~Motschan-Armen]{Annika Lang and Ioanna Motschan-Armen}
\address{Department of Mathematical Sciences,\\ Chalmers University of Technology \& University of Gothenburg, \\ 41296~Gothenburg, Sweden}              
\email{\tt annika.lang@chalmers.se, \tt ioannamo@chalmers.se}

\thanks{Acknowledgment: This work was supported in part by the Swedish Research Council (VR) through grant no.\ 2020-04170, by the Wallenberg AI, Autonomous Systems and Software Program (WASP) funded by the Knut and Alice Wallenberg Foundation, by the Chalmers AI Research Centre (CHAIR), and by the European Union (ERC, StochMan, 101088589).}


\begin{document}

\begin{abstract}
The stochastic heat equation on the sphere driven by additive isotropic Wiener noise is approximated by a spectral method in space and forward and backward Euler--Maruyama schemes in time. The spectral approximation is based on a truncation of the series expansion with respect to the spherical harmonic functions. Optimal strong convergence rates for a given regularity of the initial condition and driving noise are derived for the Euler--Maruyama methods. Besides strong convergence, convergence of the expectation and second moment is shown, where the approximation of the second moment converges with twice the strong rate. Numerical simulations confirm the theoretical results.
\end{abstract}

\maketitle


\section{Introduction}\label{sec-intro}

While stochastic partial differential equations (SPDEs) and their numerical approximations have mainly been considered in Euclidean space so far, applications motivate to extend the theory to surfaces and especially the sphere. Examples are uncertain evolution on the Earth or cells.
Numerical methods for SPDEs have been developed and analyzed for more than two decades by now, with references for example summarized in the monographs \cite{Kloeden2011,MR3308418}, but the literature on surfaces is still rare. We are only aware of the results on the sphere given in \cite{alodat2022approximation, cohen2021numerical, LeGia2019, Lang_2015, LeGiaPeach2019}.

To give this area a new push, we consider the stochastic heat equation
\begin{equation*}
\dd X(t) = \Delta_{\mathbb{S}^2} X(t) \, \dd t + \dd W(t)
\end{equation*}
on the unit sphere~$\IS^2$ with initial condition $X(0) = X_0 \in L^2(\Omega; L^2(\mathbb{S}^2))$ driven by an additive isotropic $Q$-Wiener process~$W$.

A spectral method including strong convergence for this equation has been considered in~\cite{Lang_2015} that allows only for simulation if the stochastic convolution is computed directly with the correct distribution. It does not allow to simulate solutions for a given sample path of the $Q$-Wiener process.

In this work we allow for computations based on samples of the $Q$-Wiener process by a time approximation with a forward and backward Euler--Maruyama scheme. Optimal rates for given regularity of the initial condition and noise are derived in the semigroup framework in~\cite{YubinY,Kru14} based on estimates for deterministic PDEs in~\cite{thomee97}. We are following the Gothenburg tradition of optimal estimates and derive optimal rates for strong convergence but allow for up to $\operatorname{O}(h)$ for a time step size~$h$ instead of the usually shown limit of~$\operatorname{O}(h^{1/2})$.

Additionally we show convergence of the expectation and the second moment of the solution for the spectral and the Euler--Maruyama methods. While the rates for the expectation are the same as for strong convergence due to the limits of the deterministic PDE theory, we obtain twice the rate for the second moment compared to the strong convergence for a given regularity.

In our setting we are able to show all results by elementary estimates on exponential functions and their approximation. Therefore we do not require the reader to be familiar with the semigroup theory used in~\cite{thomee97} but are able to illustrate numerical analysis for SPDEs and their optimal convergence in a more elementary way.

The outline of this paper is as follows: In Section~\ref{sec-setting} we introduce the stochastic heat equation with the necessary framework, background, and its properties. Section~\ref{sec-spec} recapitulates the spectral approximation in space presented in~\cite{Lang_2015} and its strong convergence. We show additionally convergence of the expectation and the second moment of the equation. The forward and backward Euler--Maruyama methods are then presented in Section~\ref{sec-Euler}. Based on properties of the exponential function and its approximation, we prove optimal strong convergence rates and convergence of the expectation and the second moment. We conclude in Section~\ref{sec-numerics} with numerical simulations that confirm our theoretical results. Solution paths for all approximation methods are shown at \url{https://www.youtube.com/playlist?list=PLtvKza5x5KGN6FR5JPOey85VpdJLEeY-w}. Details on the expectation and the second moment are included in Appendix~\ref{app:prop_sol} and the proofs on the estimates of the exponential functions are shown in Appendix~\ref{app:reg_exp_approx}.

\section{The stochastic heat equation on the sphere and its properties}\label{sec-setting}

We consider the stochastic heat equation on the sphere on a complete filtered probability space $(\gO, \cA, (\cF_t)_t, \IP)$ and a finite time interval $\mathbb{T} = [0,T]$, $T < + \infty$,
\begin{equation}
	\label{eq: SHEshort}
	\dd X(t) = \Delta_{\mathbb{S}^2} X(t) \, \dd t + \dd W(t)
\end{equation}
with $\cF_0$-measurable initial condition $X(0) = X_0 \in L^2(\Omega;L^2(\mathbb{S}^2))$. The equation is driven by an additive isotropic $Q$-Wiener process~$W$, i.e., $W$ is a $L^2(\IS^2)$-valued Wiener process with space covariance described by the operator~$Q$.
Before elaborating on the noise and deriving a solution $X(t) \in L^2(\gO;L^2(\IS^2))$ for the equation, let us introduce all necessary notation.

Let $\mathbb{S}^2$ denote the unit sphere in $\mathbb{R}^3$, i.e.,
\begin{equation*}
	\mathbb{S}^2 = \{ x \in \mathbb{R}^3, \norm x = 1 \},
\end{equation*}
where $\norm \cdot  $ denotes the Euclidean norm, and we equip it with the geodesic metric given by
\begin{equation*}
	d(x,y) = \arccos \, \langle x,y \rangle_{\mathbb{R}^3}
\end{equation*}
for all $x,y \in \mathbb{S}^2$. Furthermore we denote by $\sigma$ the Lebesgue measure on the sphere which admits the representation 
\begin{equation*}
	\dd \sigma (y) = \sin \vartheta \, \dd \vartheta \, \dd \varphi
\end{equation*}
for Cartesian coordinates~$y \in \mathbb{S}^2$ coupled to polar coordinates $(\vartheta, \varphi) \in [0, \pi] \times [0, 2 \pi)$ via the transformation $y = (\sin \vartheta \cos \varphi, \sin \vartheta \sin \varphi, \cos \vartheta)$.

To characterize the driving noise~$W$ and give properties of the Laplace--Beltrami operator~$\Delta_{\IS^2}$, it is essential to introduce the set of spherical harmonic functions $\cY := (Y_{\ell,m}, \, \ell \in \mathbb{N}_0, \, m = - \ell, \ldots , \ell)$ consisting of $Y_{\ell,m} : [0,\pi] \times [0,2 \pi) \rightarrow \mathbb{C}$ given by
\begin{equation*}
	Y_{\ell,m}(\vartheta,\varphi) = \sqrt{\frac{2 \ell +1}{4 \pi} \frac{(\ell - \mu)!}{(\ell + \mu)!}} \mathcal{P}_{\ell,m}(\cos \vartheta) e^{i m \varphi}
\end{equation*}
for $\ell \in \mathbb{N}_0$, $m =  0, \ldots , \ell$, and by
\begin{equation}
	\label{eq: Ylm+-}
	Y_{\ell, m} = (-1)^m \overline{Y_{\ell, -m}}
\end{equation}
for $m = - \ell , \ldots , - 1$.
Here the associated Legendre polynomials $(\mathcal{P}_{\ell,m}(\mu), \ell \in \mathbb{N}_0, m = 0, \ldots ,\ell )$ are defined by
\begin{equation*}
	\mathcal{P}_{\ell,m}(\mu) = (- 1)^m(1-\mu^2)^\frac{m}{2} \frac{\partial^m}{\partial \mu^m} \mathcal{P}_{\ell} (\mu)
\end{equation*}
for $\ell \in \mathbb{N}_0$, $m = 0,\ldots , \ell$, and $\mu \in [-1,1]$, which are themselves characterized by the Legendre polynomials $(P_{\ell}, \ell \in \mathbb{N}_0)$ that can for example be written by Rodrigues' formula (see, e.g., \cite{szego1975})
\begin{equation*}
	P_{\ell}(\mu) = 2^{- \ell} \frac{1}{\ell !} \frac{\partial^{\ell}}{\partial \mu^{\ell}} (\mu^2-1)^{\ell}
\end{equation*}
for all $\ell \in \mathbb{N}_0$ and $\mu \in [- 1, 1]$.

The spherical harmonic functions form an orthonormal basis of~$L^2(\IS^2;\C)$ and its subspace~$L^2(\IS^2)$ of all real-valued functions consists of all functions $f = \sum_{\ell =  0}^\infty \sum_{m=-\ell}^\ell f_{\ell,m} Y_{\ell,m}$ with coefficients $f_{\ell,m} \in \C$ satisfying
\begin{equation}\label{eq:real_functions}
	f_{\ell, m} = (-1)^m \overline{f_{\ell, -m}}
\end{equation}
similarly to the well-known properties of Fourier expansions of real-valued functions on~$\R$. With a slight abuse of notation we switch in what follows between Cartesian and polar coordinates and set
\begin{equation*}
	Y_{\ell,m}(y) = Y_{\ell,m}(\vartheta,\varphi)
\end{equation*} 
with $y = (\sin \vartheta \cos \varphi, \sin \vartheta \sin \varphi, \cos \vartheta)$.

We define the \emph{Laplace--Beltrami operator} or \emph{spherical Laplacian} in terms of spherical coordinates similarly to \cite[Section~3.4.3]{article2} by 
\begin{equation*}
	\Delta_{\mathbb{S}^2} = \left( \sin \vartheta \right)^{- 1} \frac{\partial}{\partial \vartheta} \left( \sin \vartheta \frac{\partial}{\partial \vartheta} \right) + \left( \sin \vartheta \right)^{- 2} \frac{\partial^2}{\partial \varphi^2}.
\end{equation*}
It is well known that it satisfies (see, e.g., Theorem~2.13 in \cite{morimoto1998analytic})
\begin{equation*}
	\Delta_{\mathbb{S}^2} Y_{\ell,m} = - \ell(\ell +1) Y_{\ell,m}
\end{equation*}
for all $\ell \in \mathbb{N}_0, \, m = - \ell, \ldots , \ell$, i.e., the spherical harmonic functions~$\mathcal{Y}$ are eigenfunctions of $\Delta_{\mathbb{S}^2}$ with eigenvalues $(- \ell ( \ell +1), \, \ell \in \mathbb{N}_0)$.

On the unit sphere we define the Sobolev spaces~$H^s(\IS^2)$ with smoothness index~$s \in \R$ via Bessel potentials as
\begin{equation*}
	H^s(\mathbb{S}^2) = (\mathrm{Id} - \Delta_{\mathbb{S}^2} )^{- s/2} L^2(\mathbb{S}^2),
\end{equation*}
with inner products given by
\begin{align*}
	\langle f,g \rangle_{H^s(\mathbb{S}^2)} = \langle (\mathrm{Id} - \Delta_{\mathbb{S}^2} )^{s/2} f,  (\mathrm{Id} - \Delta_{\mathbb{S}^2} )^{s/2} g \rangle_{L^2(\mathbb{S}^2)}.
\end{align*}
For further details on these spaces we refer for instance to~\cite{STRICHARTZ198348}.
The corresponding Lebesgue--Bochner spaces for $p \ge 1$ are denoted by $L^p(\Omega;H^s(\mathbb{S}^2))$ with norm
\begin{align*}
	\norm Z _{L^p(\Omega;H^s(\mathbb{S}^2))} = \E [ \|Z \|_{H^s(\mathbb{S}^2)}^p ]^{1/p}.
\end{align*}

The last thing to introduce from~\eqref{eq: SHEshort} before being able to solve it is the driving noise. Similarly to~\cite{Lang_2015} and~\cite{cohen2021numerical}, we introduce an isotropic $Q$-Wiener process by the series expansion, often referred to as \emph{Karhunen--Loève expansion},
\begin{align}
	\label{eqn: KarhLo}
	\begin{split}
		W(t,y) &= \sum_{\ell = 0}^{\infty} \sum_{m= - \ell}^{\ell}  a_{\ell,m}(t) Y_{\ell,m}(y) \\
		&= \sum_{\ell = 0}^{\infty} \Bigl( \sqrt{A_{\ell}} \beta_{\ell,0}^1(t) Y_{\ell,0}(y) + \sqrt{2 A_{\ell}} \sum_{m=1}^{\ell} (\beta_{\ell,m}^1(t) \mathrm{Re} Y_{\ell,m}(y) + \beta_{\ell,m}^2(t) \mathrm{Im} Y_{\ell,m}(y))\Bigr),
	\end{split}
\end{align}
where $((\beta_{\ell,m}^1,\beta_{\ell,m}^2), \, \ell \in \mathbb{N}_0, \, m = 0 , \ldots , \ell)$ is a sequence of independent, real-valued Brownian motions with $\beta_{\ell,0}^2 = 0$ for $\ell \in \mathbb{N}_0$ and $(A_\ell, \ell \in \N_0)$ denotes the \emph{angular power spectrum}.  In the equality we used the properties~\eqref{eq: Ylm+-} and~\eqref{eq:real_functions} to switch between a complex-valued and real-valued expansion. The covariance operator~$Q$ is characterized by its eigenexpansion (see, e.g., \cite{Lang2013,Lang_2015}) given by
\begin{align*}
	Q Y_{\ell,m}= A_\ell Y_{\ell,m}.
\end{align*}
The regularity of~$W$ is given by the properties of~$Q$, which in turn are described by the decay of the angular power spectrum. More specifically
\begin{align*}
	\|W(t)\|_{L^2(\Omega;H^s(\IS^2))}^2
		& = \|(\mathrm{Id} - \Delta_{\mathbb{S}^2} )^{s/2} W(t)\|_{L^2(\Omega;L^2(\IS^2))}^2
		= t \sum_{\ell = 0}^\infty (2\ell + 1) A_\ell (1+\ell(\ell+1))^s\\
		& = t \trace ((\mathrm{Id} - \Delta_{\mathbb{S}^2} )^s Q),
\end{align*}
which follows with similar calculations as in \cite[Proposition~5.2]{Lang_2015}. This expression is finite if $A_\ell \le C \ell^{-\ga}$ with $\alpha > 2(s+1)$ for all $\ell \ge \ell_0$.

We are now in state to solve the stochastic heat equation~\eqref{eq: SHEshort} which reads in integral form
\begin{align*}
	X(t) = X_0 +\int_0^t \Delta_{\mathbb{S}^2} X(s) \, \dd s + \int_0^t \, \dd W(s)
	= X_0 +\int_0^t \Delta_{\mathbb{S}^2} X(s) \, \dd s + W(t).
\end{align*}
Since the spherical harmonics are an eigenbasis of~$\Delta_{\IS^2}$ and~$Q$, we expand both sides in~$\cY$ and obtain
\begin{align}\label{eq:SHE_expansion}
	\begin{split}
	\sum_{\ell = 0}^{\infty} \sum_{m= -\ell}^{\ell} X_{\ell,m}(t) Y_{\ell,m}
	& \quad = \sum_{\ell = 0}^{\infty} \sum_{m= -\ell}^{\ell} X_{\ell,m}^0 Y_{\ell,m} + \int_0^t X_{\ell,m}(s) \Delta_{\mathbb{S}^2} Y_{\ell,m} \, \dd s +  a_{\ell,m}(t) Y_{\ell,m} \\
	& \quad = \sum_{\ell = 0}^{\infty} \sum_{m= -\ell}^{\ell} \left(X_{\ell,m}^0 - \ell(\ell+1) \int_0^t X_{\ell,m}(s) \, \dd s +  a_{\ell,m}(t) \right) Y_{\ell,m}	,	
	\end{split}
\end{align}
for the corresponding coefficients $X_{\ell,m}(t) = \langle X(t), Y_{\ell,m} \rangle_{L^2(\mathbb{S}^2;\C)}$ of the series expansion. The solution is then given by the solutions $(X_{\ell,m}, \ell \in \N_0, m=-\ell,\ldots,\ell)$ to the system of Ornstein--Uhlenbeck processes
\begin{equation}
	\label{eq: SODE}
	X_{\ell,m}(t) = X_{\ell,m}^0 - \ell ( \ell +1) \int_0^t X_{\ell,m}(s) \, \dd s +  a_{\ell,m}(t),
\end{equation}
which are obtained by the variations of constants formula
\begin{equation}
	\label{eq:SODE_solution}
	X_{\ell , m}(t) = e^{ - \ell (\ell +1)t} X_{\ell,m}^0 + \int_0^t e^{- \ell (\ell +1)(t-s)} \dd a_{\ell,m}(s).
\end{equation}

In order to simulate real-valued solutions in later sections using the expansion~\eqref{eqn: KarhLo}, we need to reformulate the equations in the real and imaginary part. Using~\eqref{eq:real_functions} and noting that $X_{\ell,0}$ and $Y_{\ell,0}$ are real-valued for all $\ell \in \N_0$, we obtain
\begin{equation}
	\label{eqn: SDE}
	\sum_{\ell = 0}^{\infty} \sum_{m= -\ell}^{\ell} X_{\ell,m}(t) Y_{\ell,m}
	 = \sum_{\ell =  0}^\infty \bigl( X_{\ell,0}(t) Y_{\ell,0} + \sum_{m=1}^{\ell} 2 \Re(X_{\ell,m}(t)) \Re (Y_{\ell,m}) - 2 \Im (X_{\ell,m}(t))  \Im (Y_{\ell,m}) \bigr).
\end{equation}
This yields for our system of stochastic differential equations~\eqref{eq: SODE} using~\eqref{eqn: KarhLo}
\begin{align}
	\label{eqn: SDEParts}
	\begin{cases}		
		X_{\ell,0}(t) = X_{\ell,0}^0 - \ell (\ell +1) \int_0^{t} X_{\ell,0}(s) \, \dd s + \sqrt{A_{\ell}} \beta_{\ell,0}^1(t), \\
		\Re(X_{\ell,m}(t)) = \Re(X_{\ell,m}^0) - \ell (\ell +1) \int_0^t \Re(X_{\ell,m}(s)) \, \dd s  + \sqrt{2^{-1}A_{\ell}} \, \beta_{\ell,m}^1(t), \\
		\Im(X_{\ell,m}(t)) = \Im(X_{\ell,m}^0) - \ell (\ell +1) \int_0^t \Im(X_{\ell,m}(s)) \, \dd s + \sqrt{2^{-1} A_{\ell}} \, \beta_{\ell,m}^2(t).
	\end{cases}
\end{align}

By straightforward computations, which we add for completeness in Appendix~\ref{app:prop_sol}, we obtain that the expectation of the solution is given by
\begin{equation}\label{eq:exact_expectation}
	\E[X(t)]
	= \sum_{\ell =  0}^{\infty} \sum_{m = - \ell}^{\ell} e^{- \ell (\ell +1)t} \E[X_{\ell,m}^0] Y_{\ell,m},
\end{equation}
and the second moment satisfies
\begin{equation*}
		\E[\|X(t)\|^2_{L^2(\IS^2)}] 
		= \sum_{\ell =  0}^{\infty} \Bigl( \sum_{m = -\ell}^{\ell} e^{-2 \ell (\ell +1)t} \E[|X_{\ell,m}^0|^2] \| Y_{\ell,m}\|^2_{L^2(\mathbb{S}^2;\C)} \Bigr)
		+  A_\ell \frac{1+ 2\ell}{2\ell(\ell+1)} (1 - e^{-2 \ell (\ell+1)t}).
\end{equation*}

\section{Spectral approximation in space}\label{sec-spec}

We start with the approximation in space by the spectral method used in~\cite{Lang_2015}. We recall the strong convergence and derive the error in the expectation and second moment.

We approximate the solution by truncating the series expansion~\eqref{eq:SHE_expansion} with the given solutions~\eqref{eq:SODE_solution} at a given $\gk > 0$, i.e., we set
\begin{equation}\label{eq:spectral_approx}
	X^{(\kappa)}(t) 
		= \sum_{\ell = 0}^{\kappa} \sum_{m = -\ell}^\ell 
			\Bigl( e^{-\ell(\ell+1)t} X_{\ell,m}^0 + \int_0^t e^{-\ell(\ell+1)(t-s)} \, \dd a_{\ell,m}(s) \Bigr) Y_{\ell,m}.
\end{equation}

Analogously to the calculations in Appendix~\ref{app:prop_sol} we derive the expectation
\begin{equation}\label{eq:exact_spectral}
 \E[X^{(\kappa)}(t)]
 	= \sum_{\ell =  0}^{\kappa} \sum_{m = - \ell}^{\ell} e^{- \ell (\ell +1)t} \E[X_{\ell,m}^0] Y_{\ell,m},
\end{equation}
 and the second moment of the spectral approximation
 \begin{align}
 \label{eq:second_moment_spectral}
 	\begin{split}
  		& \E[\|X^{(\kappa)}(t)\|^2_{L^2(\IS^2)}] \\
 			& \quad = \sum_{\ell =  0}^{\gk} \Bigl( \sum_{m = -\ell}^{\ell} \bigl( e^{-2 \ell (\ell +1)t} \E[|X_{\ell,m}^0|^2] \| Y_{\ell,m}\|^2_{L^2(\mathbb{S}^2;\C)} \bigr)
 				+  A_\ell \frac{1+ 2\ell}{2\ell(\ell+1)} (1 - e^{-2 \ell (\ell+1)t})\Bigr).	
 	\end{split}
 \end{align}

Strong convergence of the spectral approximation was already shown in Lemma~7.1 in~\cite{Lang_2015}. We state the result here with respect to the initial condition which is of interest in the next section. The constants follow immediately from the proof in~\cite{Lang_2015}.

\begin{lemma}\label{lem:Lp_conv_stoch_heat}
	Let $t \in \IT$. 
	Furthermore assume that there exist $\ell_0 \in \N$, $\ga > 0$, and a constant~$C>0$ such that the angular power spectrum $(A_\ell, \ell \in \N_0)$ satisfies $A_\ell \le C \cdot \ell^{-\ga}$ for all $\ell > \ell_0$. 
	Then the strong error of the approximate solution $X^{(\kappa)}$ is bounded uniformly in time and independently of a time discretization by
	\begin{equation*}
	\|X(t) - X^{(\kappa)}(t)\|_{L^2(\gO;L^2(\IS^2))}
	\le e^{-(\gk+1)(\gk+2)t} \|X_0\|_{L^2(\gO;L^2(\IS^2))} + \hat{C} \cdot \gk^{-\ga/2}
	\end{equation*}
	for all $\gk \ge \ell_0$ and a constant~$\hat{C}$ depending on $C$ and~$\ga$. 
\end{lemma}

We continue with the convergence of the expectation and the second moment of the equation. Since the solution is Gaussian conditioned on the initial condition, these are important quantities to characterize the solution.

\begin{lemma}\label{lem:weak_conv_stoch_heat}
	Let $t \in \IT$. 
	Furthermore assume that there exist $\ell_0 \in \N$, $\ga > 0$, and a constant~$C>0$ such that the angular power spectrum $(A_\ell, \ell \in \N_0)$ satisfies $A_\ell \le C \cdot \ell^{-\ga}$ for all $\ell > \ell_0$. 
	Then the expectation of the approximate solution $X^{(\kappa)}$ is bounded for all $\gk \ge \ell_0$ uniformly in time and independently of a time discretization by
	\begin{equation*}
	\|\E[X(t)] - \E[X^{(\kappa)}(t)]\|_{L^2(\IS^2)}
	\le e^{-(\gk+1)(\gk+2)t} \|\E[X_0]\|_{L^2(\IS^2)}.
	\end{equation*}
	The error of the second moment is bounded by
	\begin{equation*}
		| \E[ \| X(t) \|^2_{L^2(\IS^2)} - \| X^{(\kappa)}(t) \|^2_{L^2(\IS^2)}] | 
	\leq 2 \cdot e^{- 2(\gk+1)(\gk+2)t} \, \|X_0\|_{L^2(\gO;L^2(\IS^2))}^2 
		+ \hat{C} \cdot \kappa^{- \alpha}
	\end{equation*}
	for all $\gk \ge \ell_0$, where $\hat{C}$ depends on $C$ and~$\ga$.
\end{lemma}

\begin{proof}
	Given the exact formulation of the expectation of the solution~\eqref{eq:exact_expectation}, the error is given by
	\begin{equation*}
	 \|\E[X(t)] - \E[X^{(\kappa)}(t)]\|_{L^2(\IS^2)}
	 	= \Bigl\| \sum_{\ell =  \kappa + 1}^{\infty} \sum_{m = - \ell}^{\ell} e^{- \ell (\ell +1)t} \, \E[X_{\ell,m}^0] Y_{\ell,m} \Bigr\|_{L^2(\IS^2)},
	\end{equation*}
	which is bounded in the same way as in Lemma~\ref{lem:Lp_conv_stoch_heat} (see~\cite{Lang_2015}) by
	\begin{align*}
	 \Bigl\| \sum_{\ell =  \kappa + 1}^{\infty} \sum_{m = - \ell}^{\ell} e^{- \ell (\ell +1)t} \, \E[X_{\ell,m}^0] Y_{\ell,m} \Bigr\|_{L^2(\IS^2)}^2
	 	& = \sum_{\ell=\gk+1}^\infty \sum_{m= - \ell}^\ell e^{- 2\ell(\ell+1)t} \, \|\E[X_{\ell,m}^0] Y_{\ell, m}\|_{L^2(\IS^2;\C)}^2\\
	 	& \le e^{- 2(\gk+1)(\gk+2)t} \|\E[X_0]\|_{L^2(\IS^2)}^2.
	\end{align*}
	This finishes the proof of the first part of the lemma.
	
	Using the same computation as in the proof of \cite[Proposition~4]{cohen2021numerical}, one obtains for the second moment
	\begin{equation*}
	 | \E[ \| X(t) \|^2_{L^2(\IS^2)} - \| X^{(\kappa)}(t) \|^2_{L^2(\IS^2)}] |
	 	= \|X(t) - X^{(\kappa)}(t)\|_{L^2(\gO;L^2(\IS^2))}^2,
	\end{equation*}
	and applying Lemma~\ref{lem:Lp_conv_stoch_heat} yields the claim.
\end{proof}

 Having convergence results for the semidiscrete approximation at hand, we are now ready to look at time discretizations and fully discrete approximations in the next section.

\section{Euler--Maruyama approximation in time}\label{sec-Euler}

We have seen in the previous section that we can approximate the solution to~\eqref{eq: SHEshort} by the spectral approximation~\eqref{eq:spectral_approx}. Computations are only possible in practice if simulating the stochastic convolutions directly. Since we know the distribution of the stochastic convolutions, this can be done (see \cite{Lang_2015} for details). If we want to simulate the solution for a given sample of the $Q$-Wiener process~$W$, we need to take another approach. In this section we introduce forward and backward Euler--Maruyama schemes based on samples of~$W$ and show their convergence.

Let $0 = t_0 < t_1 < \ldots < t_n = T, \, \, n \in \mathbb{N}$, be an equidistant time grid with step size~$h$. The \emph{forward Euler} approximation of the exponential function $e^{-\ell(\ell+1)h}$ is given by
\begin{equation*}
	\xi = (1 - \ell (\ell +1)h).
\end{equation*}
In the later convergence analysis, we will need properties of this approximation that separate the behavior of growing~$\ell$ and $h$ going to zero. These estimates have been shown in the abstract semigroup framework, e.g., in~\cite{Kru14} and based on~\cite{thomee97}. We are able to show these optimal regularity results based on elementary computations. Surprisingly, we did not find them in the literature for finite-dimensional SODE systems, where the growth in~$\ell$ is hidden in global constants. The proof of the following proposition is given in Appendix~\ref{app:reg_exp_approx}.
\begin{proposition}\label{prop:exp_fEM}
	The exponential function and its approximation by the forward Euler approximation satisfy the following properties:
	\begin{enumerate}[label=\alph*)]
		\item \label{prop:exp_fEM_1} For all  $\mu \in (0,1]$, there exists a constant $C_\mu >0$ such that for all $\ell \in \N$ and $h > 0$
		\begin{equation*}
			| e^{-\ell(\ell+1)h} - (1-\ell(\ell+1)h) |
			\le C_\mu (\ell(\ell+1))^{1+ \mu} h^{1+ \mu}.
		\end{equation*}
		\item \label{prop:exp_fEM_3} For all  $\mu \in (0,1]$, there exists a constant $C_\mu >0$ such that for all $\ell, k \in \N$ and $h > 0$ with $\ell(\ell+1)h \le 1$
		\begin{align*}
			| e^{- \ell ( \ell +1) h\cdot k} - (1-\ell (\ell +1)h)^k |
			& \le C_\mu (\ell(\ell+1))^{1+ \mu} h^{1+ \mu} \, k \, e^{-\ell(\ell+1)h\cdot (k-1)}\\
			&\le C_\mu (\ell(\ell+1))^{\mu} \, h^{\mu}.
		\end{align*}
	\end{enumerate}
\end{proposition}

Following \cite[Definition~10]{L10}, stability is guaranteed if there exists $K \geq 1$ such that for all $h > 0$ and all $\ell \in \mathbb{N}_0$
\begin{equation*}
	| 1 - \ell ( \ell +1) h | \leq K.
\end{equation*}
Therefore this forward approximation will only lead to a stable scheme if
\begin{equation*}
	h \leq | \ell ( \ell +1) |^{-1},
\end{equation*}
which restricts the time step size~$h$ by the truncation index~$\gk$.

The \emph{backward Euler} approximation of the exponential function $e^{-\ell(\ell+1)h}$ is given by
\begin{equation*}
	\xi = (1 + \ell (\ell +1)h)^{-1},
\end{equation*}
which is unconditionally stable since
\begin{equation*}
	\left| (1 + \ell (\ell +1)h)^{-1} \right| \leq K
\end{equation*}
for any $K \geq 1$.

We prove analogous results to Proposition~\ref{prop:exp_fEM} also for the backward scheme in Appendix~\ref{app:reg_exp_approx}, which are stated in the following proposition.
\begin{proposition}\label{prop:exp_bEM}
	The exponential function and its approximation by the backward Euler approximation satisfy the following properties:
	\begin{enumerate}[label=\alph*)]
		\item \label{prop:exp_bEM_1} For all  $\mu \in (-1,1]$, there exists a constant $C_\mu >0$ such that for all $\ell \in \N$ and $h > 0$
		\begin{equation*}
			| e^{-\ell(\ell+1)h} - (1+\ell(\ell+1)h)^{-1} |
			\le C_\mu (\ell(\ell+1))^{1+\mu} h^{1+\mu}.
		\end{equation*}
		\item \label{prop:exp_bEM_3} For all  $\mu \in (-1,1]$, there exists a constant $C_\mu >0$ such that for all $\ell, k \in \N$ and $h > 0$ with $\ell(\ell+1)h \le C_c$
		\begin{align*}
			| e^{- \ell ( \ell +1) h\cdot k} - (1+\ell(\ell+1)h)^{-k} |
			& \le C_\mu (\ell(\ell+1))^{1+\mu} h^{1+\mu} \, k \, e^{-\ell(\ell+1)h\cdot (k-1)}\\
			& \le C_\mu (\ell(\ell+1))^{\mu} \, h^{\mu}.
		\end{align*}
	\end{enumerate}
\end{proposition}

Applying the forward and backward approximation to~\eqref{eqn: SDEParts} for $m=0$, we obtain the Euler--Maruyama method for the forward scheme
\begin{equation*}
	X_{\ell,0}^{(h)}(t_{k}) = (1-\ell(\ell+1)h) X_{\ell,0}^{(h)}(t_{k-1}) + \sqrt{A_{\ell}} \Delta \beta_{\ell,0}^1(t_{k}),
\end{equation*}
where $\Delta \beta_{\ell,0}^{1}(t_{k}) = \beta_{\ell,0}^1(t_{k}) - \beta_{\ell,0}^1(t_{k-1})$ denotes the increment of the Brownian motion.
Similarly the backward scheme is given by
\begin{equation*}
	X_{\ell,0}^{(h)}(t_{k}) = (1 + \ell (\ell +1)h)^{-1} \bigl( X_{\ell,0}^{(h)}(t_{k-1}) + \sqrt{A_{\ell}} \Delta \beta_{\ell,0}^1(t_{k})\bigr).
\end{equation*}

We write both schemes in one by
\begin{equation}
\label{eqn: EMParts}
X_{\ell,0}^{(h)}(t_{k}) = \xi X_{\ell,0}^{(h)}(t_{k-1}) + \xi^\gd \sqrt{A_{\ell}} \Delta \beta_{\ell,0}^1(t_{k}),
\end{equation}
where $\gd = 0$ in the forward scheme and $\gd = 1$ in the backward scheme.
Recursively, this leads to the representation
\begin{align}
\label{eqn: EulerPartsIter}
X_{\ell,0}^{(h)}(t_k) = \xi^k X_{\ell,0}^0 + \sqrt{A_{\ell}} \sum_{j=1}^k \xi^{k-j+\gd} \Delta \beta_{\ell,0}^1(t_j).
\end{align}
The equations for $m>0$ are obtained in the same way.

Our Euler--Maruyama approximation of~\eqref{eq: SHEshort} is given by
\begin{equation}
\label{eqn: EMtotal}
X^{(\gk,h)}(t_k) 
= \sum_{\ell=0}^{\kappa} X_{\ell,0}^{(h)}(t_k) Y_{\ell,0} 
	+ 2 \sum_{m = 1}^{\ell} \Re(X_{\ell,m}^{(h)}(t_k)) \Re(Y_{\ell,m}) - \Im(X_{\ell,m}^{(h)}(t_k))  \Im(Y_{\ell,m}). 
\end{equation}

Plugging the representation~\eqref{eqn: EulerPartsIter} into~\eqref{eqn: EMtotal}, observing that all stochastic increments have expectation zero, and rewriting the real and imaginary parts in terms of~$Y_{\ell,m}$, we derive the expectation of the Euler--Maruyama method
\begin{equation}\label{eq:exact_EM}
 \E[X^{(\gk,h)}(t_k)]
 	= \sum_{\ell =  0}^{\kappa} \sum_{m = - \ell}^{\ell} \xi^k \, \E[X_{\ell,m}^0] Y_{\ell,m}.
\end{equation}
For the second moment, we proceed similarly for the first term in~\eqref{eqn: EMParts} and use the properties of the independent stochastic increments to obtain
\begin{equation}\label{eq:second_moment_EM}
 \E[\|X^{(\gk, h)}(t_k)\|^2_{L^2(\IS^2)}]
 = \sum_{\ell =  0}^{\kappa} \Bigl( \sum_{m = - \ell}^{\ell} \xi^{2k} \E[|X_{\ell,m}^0|^2] \| Y_{\ell,m}\|^2_{L^2(\mathbb{S}^2;\C)} \Bigr) + A_\ell (1+ 2\ell) \sum_{j = 1}^k \xi^{2(k-j+\delta)} h.
\end{equation}

As a last prerequisite for our convergence analysis, we need regularity properties of exponential functions. As for the approximation properties in the previous propositions, the proof of the following results can be found in Appendix~\ref{app:reg_exp_approx}.
\begin{proposition}\label{prop:reg_exp}
	Assume that $\ell(\ell+1)h \le C_c$.
	The exponential function satisfies the following regularity estimates:
	\begin{enumerate}[label=\alph*)]
		\item \label{prop:reg_exp_2} For all $\mu \in (0,1]$, there exists a constant~$C_\mu$ such that for all $t_k > 0$
		\begin{align*}
			 \sum_{j = 1}^k \int_{t_{j-1}}^{t_j} (e^{ - \ell (\ell+1)(t_k-s)} - e^{ - \ell (\ell+1)(t_k-t_{j-1})})^2  \, \dd s 
			 \le C_\mu (\ell(\ell+1))^{2\mu-1} h^{2\mu}.
		\end{align*}
		\item \label{prop:reg_exp_3} For all $\mu \in (0,1]$, there exists a constant~$C_\mu$ such that for all $t_k > 0$
		\begin{align*}
			 \sum_{j = 1}^k \int_{t_{j-1}}^{t_j} (e^{ - \ell (\ell+1)(t_k-s)} - e^{ - \ell (\ell+1)(t_k-t_j)})^2  \, \dd s 
			  \le C_\mu (\ell(\ell+1))^{2\mu-1} h^{2\mu}.
		\end{align*}		
		\item \label{prop:reg_exp_4} For all $\mu \in [0,1]$, there exists a constant~$C_\mu$ such that for all $t_k > 0$
		\begin{align*}
			 \Bigl| \sum_{j = 1}^k \int_{t_{j-1}}^{t_j} e^{ - 2\ell (\ell+1)(t_k-s)} - e^{ - 2\ell (\ell+1)(t_k-t_{j-1})}  \, \dd s \Bigr|
			 \le C_\mu (\ell(\ell+1))^{\mu - 1} h^{\mu}.
		\end{align*}
		\item \label{prop:reg_exp_5} For all $\mu \in [0,1]$, there exists a constant~$C_\mu$ such that for all $t_k > 0$
		\begin{align*}
			 \Bigl| \sum_{j = 1}^k \int_{t_{j-1}}^{t_j} e^{- 2 \ell (\ell+1)(t_k-s)} - e^{ - 2 \ell (\ell+1)(t_k-t_j)}  \, \dd s \Bigr|
			  \le C_\mu (\ell(\ell+1))^{\mu - 1} h^{\mu}.
		\end{align*}
	\end{enumerate}
\end{proposition}

Having all basic estimates at hand, we are now ready to prove strong convergence with optimal rates for additive noise given the regularity of the initial condition and the noise.
The proofs are inspired by \cite{Kru14} but bring the semigroup theory and estimates going back to \cite{thomee97} to an elementary level.

\begin{theorem}\label{trm: strong_conv_EM}
Assume that there exist $\ga > 0$ and a constant~$C>0$ such that the angular power spectrum $(A_\ell, \ell \in \N_0)$ satisfies $A_\ell \le C \cdot \ell^{-\ga}$ for $\ell > 0$ and that $X_0 \in L^2(\gO;H^{\eta}(\IS^2))$ for some $\eta > 0$. 
Then for all $\gk \in \N$ and $h > 0$ such that $\gk(\gk+1)h \le C_c$, the strong error between $X^{(\kappa)}$ and $X^{(\gk,h)}$ is uniformly bounded for some constant~$\hat{C}$ on all time grid points~$t_k$ by
\begin{equation*}
	\| X^{(\kappa)}(t_k) - X^{(\gk,h)}(t_k) \|_{L^2(\Omega;L^2(\mathbb{S}^2))} 
	\le \hat{C} \bigl( h^{\min\{1,\eta/2\}} \|X_0\|_{L^2(\gO;H^{\eta}(\mathbb{S}^2))}  +  h^{\min\{1,\ga/4\}} \bigr).
\end{equation*}
\end{theorem}

\begin{proof}
Using the truncated version of~\eqref{eqn: SDE} and~\eqref{eqn: EMtotal}, we write the error in the real and imaginary parts as
\begin{align}
\label{eq: strongfEMtotalparts}
\begin{split}
&\| X^{(\kappa)}(t_k) - X^{(\gk,h)}(t_k) \|_{L^2(\Omega;L^2(\mathbb{S}^2))}^2 \\
& \quad =  \sum_{\ell = 1}^{\kappa} \biggl(\E \left[ | X_{\ell,0}(t_k) - X_{\ell,0}^{(h)}(t_k) |^2 \right] \|Y_{\ell,0}\|_{L^2(\IS^2)}^2 \\
& \hspace*{6em} + 2 \sum_{m = 1}^{\ell} \Bigl( \E \left[ | \Re(X_{\ell,m}(t_k)) - \Re(X_{\ell,m}^{(h)}(t_k)) |^2 \right]   \|\Re Y_{\ell,m}\|_{L^2(\IS^2)}^2\\
& \hspace*{11em} + \E \left[ | \Im(X_{\ell,m}(t_k)) - \Im(X_{\ell,m}^{(h)}(t_k)) |^2 \right] \|\Im Y_{\ell,m}\|_{L^2(\IS^2)}^2 \Bigr) \biggr).
\end{split}
\end{align}
The first difference satisfies with the formulations~\eqref{eq:SODE_solution} and~\eqref{eqn: EulerPartsIter} for $m=0$ that
\begin{align*}
& \E \left[ | X_{\ell,0}(t_k) - X_{\ell,0}^{(\kappa,h)}(t_k) |^2 \right] \\
& \quad = \E \left[ \Bigl| \bigl( e^{- \ell ( \ell +1) t_k} - \xi^k \bigr) X_{\ell,0}^0 + \sqrt{A_{\ell}} \left( \int_0^{t_k} e^{- \ell ( \ell +1)( t_k-s)} \, \dd \beta_{\ell,0}^1(s) - \sum_{j=1}^k \xi^{k-j+\gd} \Delta \beta_{\ell,0}^1(t_j) \right) \Bigr|^2 \right] \\
& \quad = \bigl( e^{- \ell ( \ell +1) t_k} - \xi^k \bigr)^2 \E \left[ | X_{\ell,0}^0 |^2 \right] + A_{\ell}\E \left[ \Bigl| \sum_{j=1}^k \int_{t_{j-1}}^{t_j} e^{- \ell ( \ell +1)( t_k-s)} - \xi^{k-j+\gd} \, \dd \beta_{\ell,0}^1(s) \Bigr|^2 \right] ,
\end{align*}
where we used that the mixed term vanishes due to the mean zero of the Gaussian increments and that $\Delta \beta_{\ell,0}^1(t_j)  = \int_{t_{j-1}}^{t_j} \, \dd \beta_{\ell,0}^1(s)$.

The first term is bounded by Proposition~\ref{prop:exp_fEM}~\ref{prop:exp_fEM_3} and Proposition~\ref{prop:exp_bEM}~\ref{prop:exp_bEM_3}, respectively, by
\begin{equation*}
	\bigl( e^{- \ell ( \ell +1) t_k} - \xi^k \bigr)^2 \E \left[ | X_{\ell,0}^0 |^2 \right]
	\le \left( C_{\eta/2} ( \ell ( \ell +1))^{\eta/2} h^{\eta/2} \right)^2 \E[ | X_{\ell,0}^0 |^2],
\end{equation*}
for $\eta \in (0,2]$. Exploiting that $( \ell ( \ell +1))^{\eta/2}\|Y_{\ell,0}\|_{L^2(\IS^2)} \le \|(\mathrm{Id} - \Delta_{\mathbb{S}^2})^{\eta/2} Y_{\ell,0}\|_{L^2(\IS^2)} = \|Y_{\ell,0}\|_{H^{\eta}(\IS^2)}$ by the definition of the norm and the eigenvalues of~$\Delta_{\IS^2}$, we obtain
\begin{align*}
	& \bigl( e^{- \ell ( \ell +1) t_k} - \xi^k \bigr)^2 \E \left[ | X_{\ell,0}^0 |^2 \right] \|Y_{\ell,0}\|_{L^2(\IS^2;\mathbb{C})}^2
	\le C_{\eta/2}^2 \, h^{\eta} \, \E[ | X_{\ell,0}^0 |^2] \|Y_{\ell,0}\|_{H^{\eta}(\IS^2)}^2.
\end{align*}

Applying the It\^o isometry to the second term yields
\begin{align*}
& \E \left[ \Bigl| \sum_{j=1}^k \int_{t_{j-1}}^{t_j} e^{- \ell ( \ell +1)( t_k-s)} - \xi^{k-j+\gd} \, \dd \beta_{\ell,0}^1(s) \Bigr|^2 \right] 
 = \sum_{j=1}^k \int_{t_{j-1}}^{t_j} \bigl( e^{- \ell ( \ell +1)( t_k-s)} - \xi^{k-j+\gd} \bigr)^2 \, \dd s \\
 & \quad \le 2 \sum_{j = 1}^k \int_{t_{j-1}}^{t_j} \bigl( e^{- \ell ( \ell +1)(t_k-s)} - e^{- \ell ( \ell +1)(t_k-t_{j-\gd})}\bigr)^2 \, \dd s + \int_{t_{j-1}}^{t_j} \left(e^{- \ell ( \ell +1)(t_k-t_{j-\gd})} - \xi^{k-j+\gd} \right)^2 \, \dd s \\
 & \quad \le 2 C_\mu (\ell(\ell+1))^{2\mu - 1} h^{2\mu} + 2 \sum_{j = 1}^k\int_{t_{j-1}}^{t_j} \bigl(e^{- \ell ( \ell +1)(t_k-t_{j-\gd})} - \xi^{k-j+\gd} \bigr)^2 \, \dd s,
\end{align*}
where we applied Proposition~\ref{prop:reg_exp}~\ref{prop:reg_exp_2} and~\ref{prop:reg_exp_3} in the last step for $\mu \in (0,1]$. Using the first inequality in Proposition~\ref{prop:exp_fEM}~\ref{prop:exp_fEM_3} and Proposition~\ref{prop:exp_bEM}~\ref{prop:exp_bEM_3} for $\mu = 1$, respectively, we bound the last term by
\begin{align*}
	& \sum_{j = 1}^k\int_{t_{j-1}}^{t_j} \left(e^{- \ell ( \ell +1)(t_k-t_{j-\gd})} - \xi^{k-j+\gd} \right)^2 \, \dd s\\
		& \quad \le  \sum_{j = 1}^k\int_{t_{j-1}}^{t_j} \left(  C_1 (\ell(\ell+1))^2 h^2 \, (k-j+\gd) \, e^{-\ell(\ell+1)h\cdot (k-j+\gd-1)} \right)^2 \, \dd s\\
		& \quad = C_1^2 (\ell(\ell+1))^4 h^2 h \sum_{j = 1}^k (h(k-j+\gd))^2 \, e^{-2\ell(\ell+1)h\cdot (k-j+\gd - 1)}.
\end{align*}
The key estimate for optimal rates with respect to the regularity of the driving noise is to bound the sum
\begin{align*}
	& h \sum_{j = 1}^k (h(k-j+\gd))^2 \, e^{-2\ell(\ell+1)h\cdot (k-j+\gd - 1)}\\
		& \quad = h \sum_{j = 0}^{k-1} (h(j+\gd))^2 \, e^{-2\ell(\ell+1)h\cdot (j+\gd - 1)}
		\le e^{2 C_c} \int_0^\infty (s+h)^2 \, e^{-2\ell(\ell+1)s} \, \dd s\\
		& \quad = e^{2 C_c} \left( \frac{h^2}{2\ell(\ell+1)} + \frac{h}{2(\ell(\ell+1))^2} + \frac{1}{4(\ell(\ell+1))^3} \right)
\end{align*}
by an integral, which holds since $\ell(\ell+1)h \le C_c$ and the integral is decaying for $s \ge \max\{1,(\ell(\ell+1))^{-1}-h\}$.
Plugging this bound in and resorting, we obtain
\begin{align*}
	& \sum_{j = 1}^k\int_{t_{j-1}}^{t_j} \left(e^{- \ell ( \ell +1)(t_k-t_{j-\gd})} - \xi^{k-j+\gd} \right)^2 \, \dd s\\
	& \quad \le C_1^2 e^{2 C_c} (\ell(\ell+1))^{2\mu-1} h^{2\mu} 
		\bigl( (h\ell(\ell+1))^{4-2\mu} + (h\ell(\ell+1))^{3-2\mu} + (h\ell(\ell+1))^{2-2\mu} \bigr) \\
		& \quad \le \tilde{C} (\ell(\ell+1))^{2\mu-1} h^{2\mu},
\end{align*}
where we used in the last inequality that $\ell(\ell+1)h \le C_c$.

The combination of the estimates on the initial condition and on the stochastic convolution yields
\begin{align*}
& \E \left[ | X_{\ell,0}(t_k) - X_{\ell,0}^{(h)}(t_k) |^2 \right] \|Y_{\ell,0}\|_{L^2(\IS^2)}^2\\
& \quad \le C_{\eta/2}^2 \, h^{\eta} \, \E[ | X_{\ell,0}^0 |^2] \|Y_{\ell,0}\|_{H^{\eta}(\IS^2)}^2
	+ 4 \tilde{C} A_\ell (\ell(\ell+1))^{2\mu-1} h^{2\mu} \|Y_{\ell,0}\|_{L^2(\IS^2)}^2.
\end{align*}
The terms for $m > 0$ are bounded in the same way.
 
Putting all parts of~\eqref{eq: strongfEMtotalparts} together, we bound
\begin{equation*}
\| X^{(\kappa)}(t_k) - X^{(\gk,h)}(t_k) \|_{L^2(\Omega;L^2(\mathbb{S}^2))}^2 
\le  C_{\eta/2}^2 \, h^{\eta} \|X_0\|_{L^2(\gO;H^{\eta}(\mathbb{S}^2))}^2	+ 4 \tilde{C} h^{2\mu} \sum_{\ell =  1}^\gk A_\ell (2\ell+1) (\ell(\ell+1))^{2\mu-1} 
\end{equation*}
and conclude with the observation that the last term satisfies
\begin{equation*}
	\sum_{\ell =  1}^\gk A_\ell (2\ell+1) ( \ell ( \ell +1))^{2\mu - 1}
	\le C \sum_{\ell =  1}^\gk \ell^{-\ga + 1 + 4\mu - 2}
	\le C \, \gk^{4\mu - \ga},
\end{equation*}
which is bounded for $\mu \le \ga/4$. Since $\mu \in (0,1]$, the claim follows.
\end{proof}

Putting Lemma~\ref{lem:Lp_conv_stoch_heat} and Theorem~\ref{trm: strong_conv_EM} together, the total error is bounded by
\begin{equation*}
	\| X(t_k) - X^{(\gk,h)}(t_k) \|_{L^2(\Omega;L^2(\mathbb{S}^2))} 
	\le \hat{C} \bigl( h^{\min\{1,\eta/2\}} \|X_0\|_{L^2(\gO;H^{\eta}(\mathbb{S}^2))}  + \gk^{-\ga/2} + h^{\min\{1,\ga/4\}} \bigr)
\end{equation*}
and the rates are balanced for $\ga = 2 \eta$.

Optimal rates for additive noise and multiplicative noise were derived in~\cite{YubinY} and~\cite{Kru14}, respectively, for convergence up to~$\operatorname{O}(h^{\min\{1,\gb\}/2})$ under the assumption that $X_0 \in L^2(\gO;H^\gb(\IS^2))$ and $\trace((-\Delta_{\IS^2})^{(\gb-1)/2}Q) < + \infty$. Setting $\gb = \eta = \ga/2$, the assumptions coincide with our conditions.

Having shown strong convergence, we continue with the time discretization error of the expectation and the second moment extending Lemma~\ref{lem:weak_conv_stoch_heat} to the fully discrete setting.

\begin{theorem} \label{trm: weak_conv_EM}
Assume that there exist $\ga > 0$ and a constant~$C>0$ such that the angular power spectrum $(A_\ell, \ell \in \N_0)$ satisfies $A_\ell \le C \cdot \ell^{-\ga}$ for all $\ell > 0$ and that $X_0 \in L^2(\gO;H^\eta(\IS^2))$ for some $\eta > 0$. 

Then for all $\gk \in \N$ and $h > 0$ such that $\gk(\gk+1)h \le C_c$, the error of the expectation is uniformly bounded for some constant~$\hat{C}>0$ on all time grid points~$t_k$ by
\begin{equation*}
\| \E[ X^{(\kappa)}(t_k) - X^{(\gk,h)}(t_k)  ]\|_{L^2(\mathbb{S}^2)} 
	\leq \hat{C} h^{\min\{1,\eta/2\}} \left\| \E[X_0] \right\|_{H^{\eta}(\mathbb{S}^2)}.
\end{equation*}

The second moment satisfies under the same assumptions that
	\begin{align*}
		\left| \E\left[ \| X^{(\kappa)}(t_k) \|^2_{L^2(\mathbb{S}^2)} - \| X^{(\gk,h)}(t_k) \|^2_{L^2(\mathbb{S}^2)} \right] \right|
		\le \hat{C} \bigl( h^{\min\{1,\eta\}} \|X_0\|_{L^2(\gO;H^{\eta}(\mathbb{S}^2))}^2
		+ h^{\min\{1,\ga/2\}} \bigr).
	\end{align*}
\end{theorem}

\begin{proof}
We observe first that
\begin{align*}
\E [ X^{(\kappa)}(t_k) ] - \E [ X^{(\gk,h)}(t_k) ]  =  \sum_{\ell = 0}^{\kappa} \sum_{m = - \ell}^{\ell} \bigl( e^{- \ell (\ell +1)t_k} - \xi^k \bigr) \E[X_{\ell,m}^0] Y_{\ell,m}
\end{align*}
using \eqref{eq:exact_spectral} and \eqref{eq:exact_EM} combined with the linearity of the expectation.

Using Proposition~\ref{prop:exp_fEM} \ref{prop:exp_fEM_3} or Proposition~\ref{prop:exp_bEM} \ref{prop:exp_bEM_3}, respectively, we bound the above by
\begin{align*}
& \| \E [ X^{(\kappa)}(t_k) ] - \E [ X^{(\gk,h)}(t_k) ] \|_{L^2(\mathbb{S}^2)}^2\\
& \quad \leq  \sum_{\ell = 1}^{\kappa} \sum_{m = - \ell}^{\ell} \bigl( C_{\eta/2} ( \ell ( \ell +1))^{\eta/2} h^{\eta/2} \bigr)^2 \, \E[ | X_{\ell,m}^0 |^2] \left\|  Y_{\ell,m}  \right\|_{L^2(\mathbb{S}^2;\C)}^2 \\
& \quad \le  \sum_{\ell = 1}^{\kappa} C_{\eta/2}^2 \, h^{\eta}\sum_{m = - \ell}^{\ell} \E[ | X_{\ell,m}^0 |^2] \| (\mathrm{Id} - \Delta_{\mathbb{S}^2} )^{\eta/2}  Y_{\ell,m}  \|_{L^2(\mathbb{S}^2;\C)}^2
\leq  C_{\eta/2}^2 \, h^{\eta}  \, \| \E[ X_0 ]\|_{H^{\eta}(\mathbb{S}^2)}^2
\end{align*}
for $\eta \in (0,2]$. Taking the square root finishes the proof of the first claim.

For the second moment, we apply \eqref{eq:second_moment_spectral} and \eqref{eq:second_moment_EM} to get
\begin{align}
\label{eq: secmom_parts_proof}
\begin{split}
& \E\left[ \| X^{(\kappa)}(t_k) \|^2_{L^2(\mathbb{S}^2)} - \| X^{(\gk,h)}(t_k) \|^2_{L^2(\mathbb{S}^2)} \right] \\
& \quad = \sum_{\ell = 1}^{\kappa} \sum_{m = - \ell}^{\ell} \bigl( e^{- 2 \ell (\ell +1)t_k} - \xi^{2k} \bigr) \E[|X_{\ell,m}^0|^2] \| Y_{\ell,m}\|^2_{L^2(\mathbb{S}^2)}  \\
& \hspace*{4em} + A_\ell (1+ 2\ell) \Bigl( (2\ell(\ell+1))^{-1} (1 - e^{- 2 \ell (\ell+1)t_k}) - \sum_{j = 1}^k \xi^{2(k-j+\gd)} h \Bigr)  \\
& \quad = \sum_{\ell = 1}^{\kappa} \sum_{m = - \ell}^{\ell} \bigl( e^{- 2 \ell (\ell +1)t_k} - \xi^{2k} \bigr) \E[|X_{\ell,m}^0|^2] \| Y_{\ell,m}\|^2_{L^2(\mathbb{S}^2)} \\
& \hspace*{4em} + A_\ell (1+ 2\ell) \Bigl( \sum_{j = 1}^k \int_{t_{j-1}}^{t_j} e^{- 2 \ell (\ell+1)(t_k-s)} - \xi^{2(k-j+\gd)}  \, \dd s  \Bigr), 
\end{split}
\end{align}
using in the last equation that
\begin{equation*}
	(2\ell(\ell+1))^{-1} (1 - e^{- 2 \ell (\ell+1)t_k}) 
		= \int_0^{t_k} e^{- 2 \ell (\ell+1)(t_k-s)} \, \dd s 
		= \sum_{j = 1}^k \int_{t_{j-1}}^{t_j} e^{- 2 \ell (\ell+1)(t_k-s)} \, \dd s .
\end{equation*}
Similarly to the proof of Theorem~\ref{trm: strong_conv_EM}, we split
\begin{equation*}
e^{- 2 \ell ( \ell +1)(t_k-s)} - \xi^{2(k-j+\gd)} 
= (e^{- 2 \ell ( \ell +1)(t_k-s)} - e^{- 2 \ell ( \ell +1)(t_k-t_{j-\gd})})
	 + (e^{- 2 \ell ( \ell +1)(t_k-t_{j-\gd})} - \xi^{2(k-j+\gd)})
\end{equation*} 
and obtain two integrals in \eqref{eq: secmom_parts_proof} which we bound separately. To the first integral we apply Proposition~\ref{prop:reg_exp}\ref{prop:reg_exp_4} or~\ref{prop:reg_exp_5}, respectively. The second one can be bounded in a similar way as the stochastic term in the proof of Theorem~\ref{trm: strong_conv_EM}. Using the first inequality in Proposition~\ref{prop:exp_fEM}\ref{prop:exp_fEM_3} or Proposition~\ref{prop:exp_bEM}\ref{prop:exp_bEM_3} for $\mu = 1$, respectively, and resorting the terms, we start with
\begin{align*}
	& \Bigl| \sum_{j = 1}^k \int_{t_{j-1}}^{t_j} e^{- 2 \ell (\ell+1)(t_k-s)} - \xi^{2(k-j+\gd)}  \, \dd s \Bigr|\\
	& \quad \le C_1 2 (\ell(\ell+1))^2 h^2 \sum_{j = 1}^k h(k-j+\gd) \, e^{-2\ell(\ell+1)h\cdot (k-j+\gd - 1)}.
\end{align*}
Again we bound the last term by the corresponding integral to obtain
\begin{align*}
	 h \sum_{j = 1}^k h(k-j+\gd) \, e^{-2\ell(\ell+1)h\cdot (k-j+\gd - 1)}
	& \le e^{2 C_c} \int_0^\infty (s+h) \, e^{-2\ell(\ell+1)s} \, \dd s\\
	& = e^{2 C_c} \left( \frac{h}{2(\ell(\ell+1))} + \frac{1}{4(\ell(\ell+1))^2} \right),
\end{align*}
since $\ell(\ell+1)h \le C_c$, and conclude using the same bound that
\begin{equation*}
	\Bigl| \sum_{j = 1}^k\int_{t_{j-1}}^{t_j} e^{- 2\ell ( \ell +1)(t_k-t_{j-\gd})} - \xi^{2(k-j+\gd)} \, \dd s \Bigr|
	\le \tilde{C} (\ell(\ell+1))^{\mu-1} h^{\mu}.
\end{equation*}

The first term in~\eqref{eq: secmom_parts_proof} is bounded using as in the proof of Theorem~\ref{trm: strong_conv_EM} Proposition~\ref{prop:exp_fEM}\ref{prop:exp_fEM_3} or Proposition~\ref{prop:exp_bEM}\ref{prop:exp_bEM_3}, respectively. 
All together we get
\begin{align*}
& \E\left[ \| X^{(\kappa)}(t_k) \|^2_{L^2(\mathbb{S}^2)} - \| X^{(\gk,h)}(t_k) \|^2_{L^2(\mathbb{S}^2)} \right] \\
& \quad \leq C_{\min\{1,\eta\}} h^{\min\{1,\eta\}} \|X_0\|_{L^2(\gO;H^{\eta}(\mathbb{S}^2))}^2
+ 2 \tilde{C} h^{\mu} \sum_{\ell = 1}^{\gk} A_\ell (2\ell+1) (\ell(\ell+1))^{\mu-1}
\end{align*}
for $\mu \in (0,1]$.

We finish the proof by observing that the last term satisfies
\begin{equation*}
	\sum_{\ell = 1}^{\gk} A_\ell (2\ell+1) (\ell(\ell+1))^{\mu-1}
		\le C \sum_{\ell = 0}^{\gk} \ell^{-\ga + 1 + 2\mu - 1}
		\le C \gk^{2\mu-\ga},
\end{equation*}
which is bounded for all $\mu \le \min\{1, \ga/2\}$.
\end{proof}

Putting together Lemma~\ref{lem:weak_conv_stoch_heat} and Theorem~\ref{trm: weak_conv_EM}, the total errors are bounded by
\begin{equation*}
	\| \E[ X(t_k) - X^{(\gk,h)}(t_k)  ]\|_{L^2(\mathbb{S}^2)} \leq C h^{\min\{1,\eta/2\}} \left\| \E[X_0] \right\|_{H^{\eta}(\mathbb{S}^2)}.
\end{equation*}
and
\begin{equation*}
	\left| \E\left[ \| X(t_k) \|^2_{L^2(\mathbb{S}^2)} - \| X^{(\gk,h)}(t_k) \|^2_{L^2(\mathbb{S}^2)} \right] \right|
	\le C \bigl( h^{\min\{1,\eta\}} \|X_0\|_{L^2(\gO;H^{\eta}(\mathbb{S}^2))}^2
	+ \gk^{-\ga} + h^{\min\{1,\ga/2\}} \bigr).
\end{equation*}
While the error in the expectation coincides with the strong error in Theorem~\ref{trm: strong_conv_EM}, due to the properties of the corresponding deterministic PDE, the error rate in the second moment is twice that of strong convergence under fixed regularity properties. We are thus able to confirm the rule of thumb that the weak rate is twice the strong one with time convergence limited by~$1$.

\section{Numerical simulation}\label{sec-numerics}

We are now ready to confirm our theoretical results from Sections~\ref{sec-spec} and~\ref{sec-Euler} with numerical experiments. We compare the convergence rates of the different errors for the spectral approximation, the forward and the backward Euler--Maruyama scheme.
\begin{figure}[htb]
  \begin{subfigure}[c]{0.3\textwidth}
  	\includegraphics[width=\textwidth]{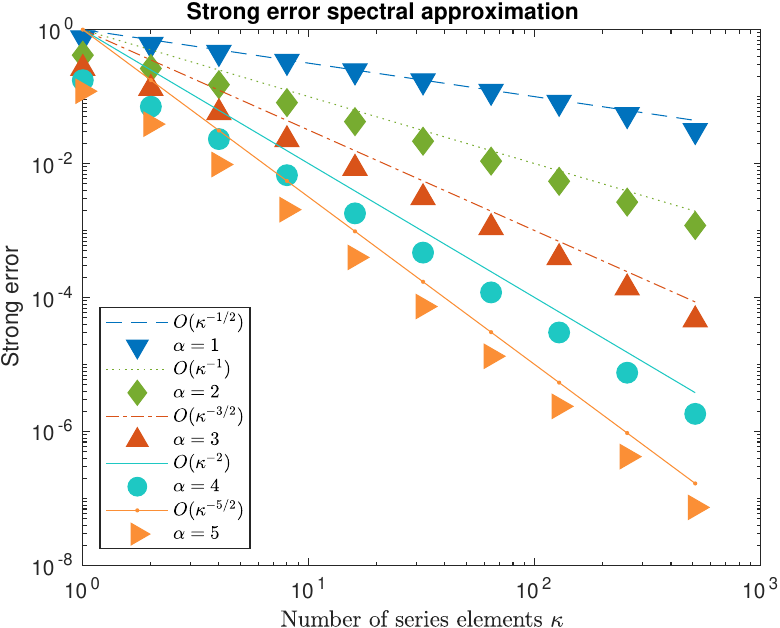}  
  	\caption{Strong error.}
    \label{subfig:f1}
  \end{subfigure}
  \hfill
  \begin{subfigure}[c]{0.3\textwidth}
  	\includegraphics[width=\textwidth]{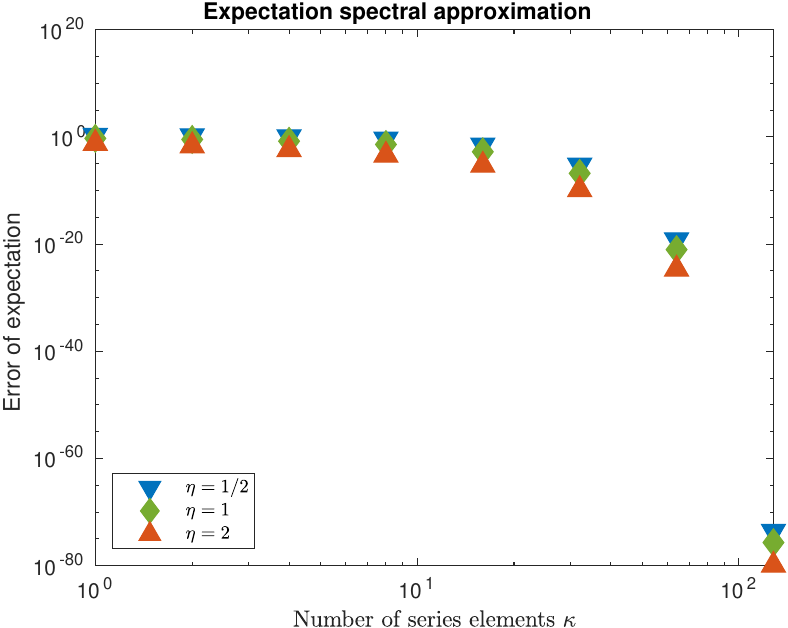}  
  	\caption{Error expectation.}
    \label{subfig:f2}
  \end{subfigure}
  \hfill
  \begin{subfigure}[c]{0.3\textwidth}
  	\includegraphics[width=\textwidth]{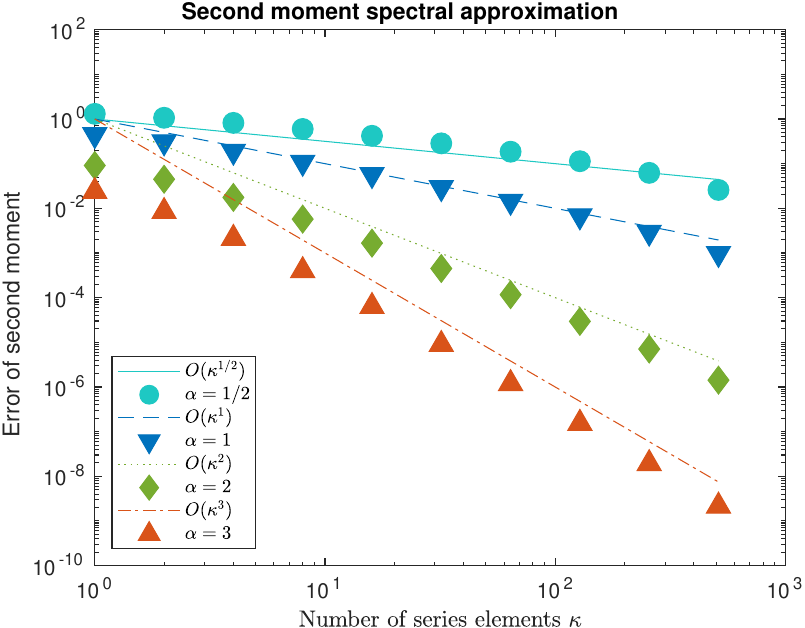} 
  	\caption{Error second moment.}
   \label{subfig:f3}
  \end{subfigure}
  \caption{Convergence of spectral approximation for different~$\ga$.}
  \label{fig:spec_appr}
\end{figure}

For the spectral approximation, we use a reference solution with $\gk = 2^{10}$ at time $T=1$ and compare it to the approximations based on $\kappa = 2^j$ for $j = 0, \ldots , 9$. In Figure~\ref{subfig:f1} we computed the expectations of the strong error explicitly while we used $10$~Monte Carlo samples in Figure~\ref{fig:f10}. The obtained rates for $\ga=1,\ldots,5$ coincide with those proven in Lemma~\ref{lem:Lp_conv_stoch_heat}. Since the error in the initial condition converges exponentially fast and we cannot see a difference in the convergence plots, we set $X_0 = 0$.

This exponential convergence is visible in Figure~\ref{subfig:f2}, which confirms the convergence of the expectation in Lemma~\ref{lem:weak_conv_stoch_heat}. Due to the fast smoothing of the solution, we use $T=0.01$. Setting $X_0=0$ and computing the expectations explicitly, we confirm the convergence rates of the second moments from Lemma~\ref{lem:weak_conv_stoch_heat} for $\ga=1/2,1,2,3$ in Figure~\ref{subfig:f3}.
\begin{figure}[htb]
\centering
  \begin{subfigure}[c]{0.3\textwidth}
  	\includegraphics[width=\textwidth]{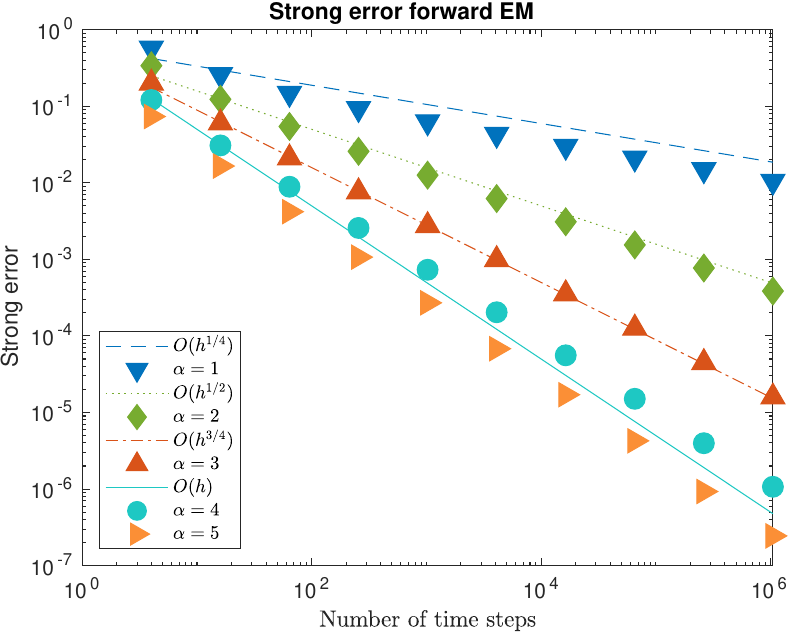}  
  	\caption{Strong error.}
    \label{subfig:f4}
  \end{subfigure}
  \hfill
  \begin{subfigure}[c]{0.3\textwidth}
  	\includegraphics[width=\textwidth]{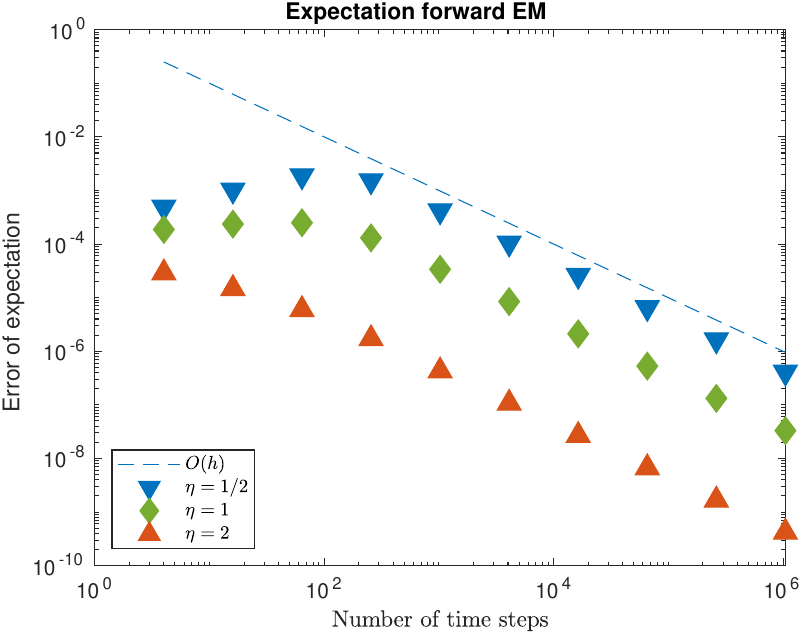}  
  	\caption{Error expectation.}
    \label{subfig:f5}
  \end{subfigure}
  \hfill
  \begin{subfigure}[c]{0.3\textwidth}
  	\includegraphics[width=\textwidth]{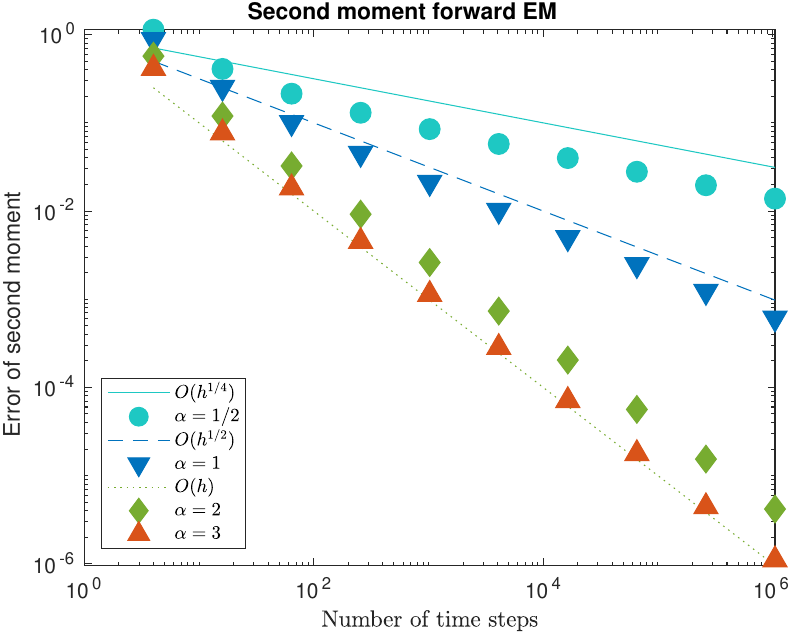} 
  	\caption{Error second moment.}
   \label{subfig:f6}
  \end{subfigure}
    \caption{Convergence of the forward Euler--Maruyama scheme with respect to the time step size~$h$ for different~$\ga$.}
    \label{fig:fEM}
\end{figure}

Having verified the spectral convergence, it remains to simulate the time discretization with the forward and backward Euler--Maruyama scheme. For that we focus on the error between $X^{(\kappa)}$ and $X^{(\gk,h)}$. We simulate on time grids with step size $h = 2^{-2\cdot m}$ for $m = 1, \ldots, 10$ coupled with $\gk = 2^m$ to guarantee stability for the forward Euler--Maruyama scheme and since larger~$\gk$ do not change the simulation results. As for the spectral approximations, we set $X_0 = 0$ to focus on the convergence with respect to the smoothness of the noise given by~$\ga$. The results for the forward Euler--Maruyama scheme in Figure~\ref{subfig:f4} using the exact expectations confirm the expected convergence of $\operatorname{O}(h^{\min\{1,\ga/4\}})$ from Theorem~\ref{trm: strong_conv_EM}.
Similar results are obtained for the backward Euler--Maruyama method in Figure~\ref{subfig:f7}. For completeness we added the corresponding results for the forward and backward scheme based on $10$~Monte Carlo samples and with reference solution using $h = 2^{-14}$ and $\gk = 2^7$ in Figures~\ref{fig:f11} and~\ref{fig:f12}.
\begin{figure}[htb]
	\centering
	\hfill
	\begin{subfigure}[c]{0.3\textwidth}
		\includegraphics[width=\textwidth]{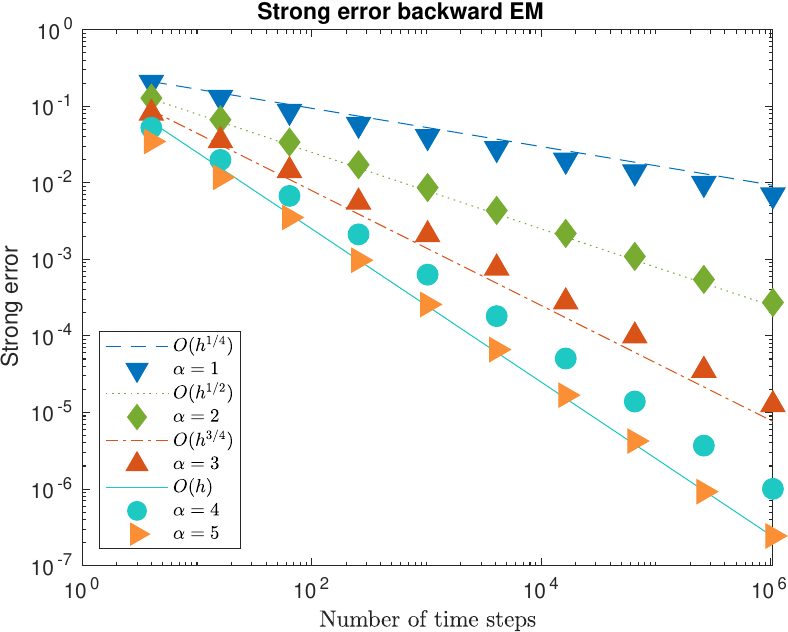}    	
		\caption{Strong error.}
		\label{subfig:f7}
	\end{subfigure}
	\hfill
	\begin{subfigure}[c]{0.3\textwidth}
		\includegraphics[width=\textwidth]{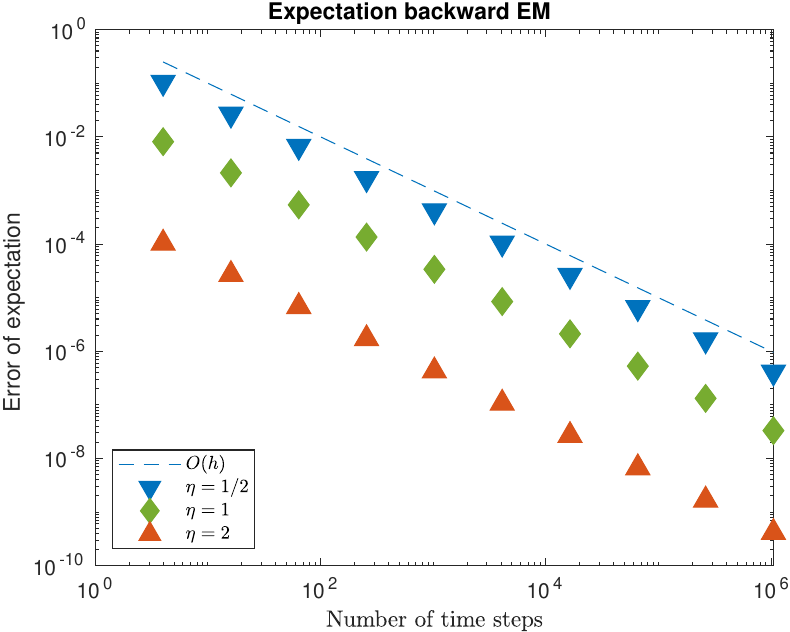}  
		\caption{Error expectation.}
		\label{subfig:f8}
	\end{subfigure}
	\hfill
	\begin{subfigure}[c]{0.3\textwidth}
		\includegraphics[width=\textwidth]{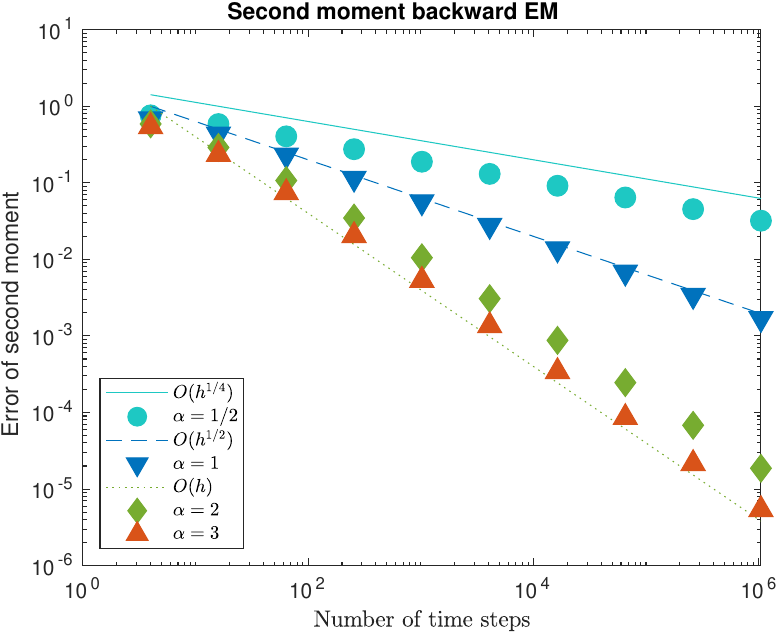} 
		\caption{Error second moment.}
		\label{subfig:f9}
	\end{subfigure}
	\caption{Convergence of the backward Euler--Maruyama scheme with respect to the time step size~$h$ for different~$\ga$.}
	\label{fig:bEM}
\end{figure}

Figure~\ref{subfig:f5} and Figure~\ref{subfig:f8} show the simulated convergence of the expectation for $\eta = 1/2, 1, 2$, where we would expect from Theorem~\ref{trm: weak_conv_EM} no convergence, convergence of rate~$1/2$ and~$1$, respectively. We used $T=0.01$ to minimize the smoothing over time. Still it is clear that all solutions are smooth for finite~$\gk$. Therefore the simulations all show $\operatorname{O}(h)$ convergence but with different error constant depending on~$\eta$.

As for the strong error, we set $X_0 = 0$ in the simulation of the error of the second moment to focus on the convergence with respect to the noise smoothness~$\ga$. In Figures~\ref{subfig:f6} and~\ref{subfig:f9} for the forward and backward Euler--Maruyama schemes, we observe convergence of $\operatorname{O}(h^{\min\{1,\ga/2\}})$, which confirms Theorem~\ref{trm: weak_conv_EM} for the second moment.
\begin{figure}[htb]
  \begin{subfigure}[t]{0.3\textwidth}
  	\includegraphics[width=\textwidth]{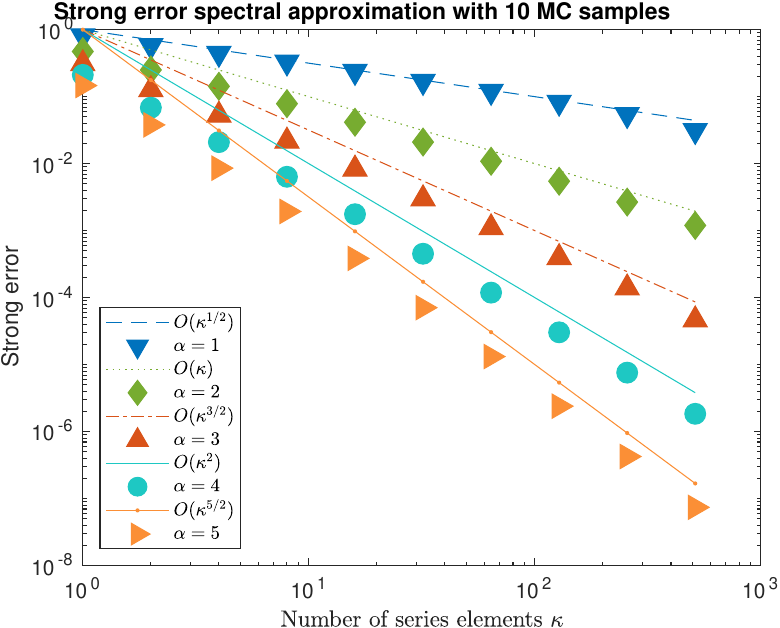}  
  	\caption{Spectral approximation.}
    \label{fig:f10}
  \end{subfigure}
  \hfill
  \begin{subfigure}[t]{0.3\textwidth}
  	\includegraphics[width=\textwidth]{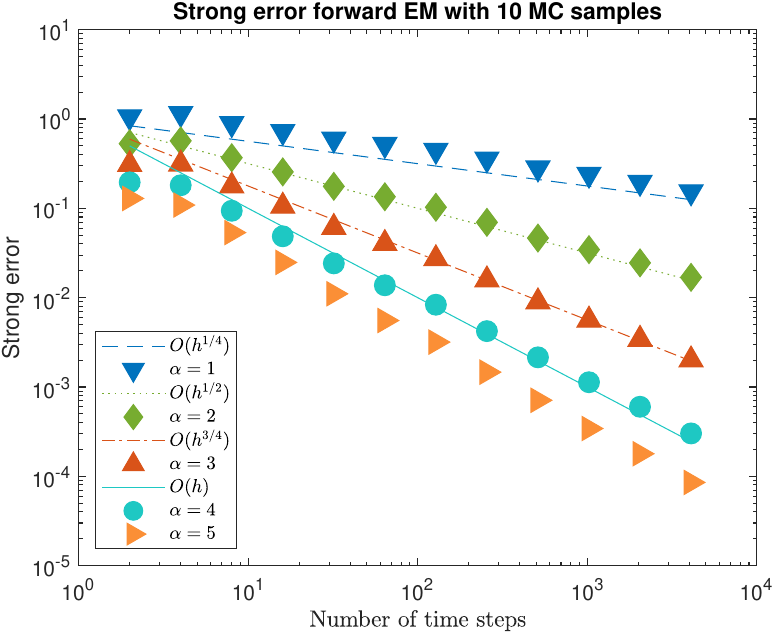}  
  	\caption{Forward Euler--Maruyama.}
    \label{fig:f11}
  \end{subfigure}
  \hfill
  \begin{subfigure}[t]{0.32\textwidth}
  	\includegraphics[width=\textwidth]{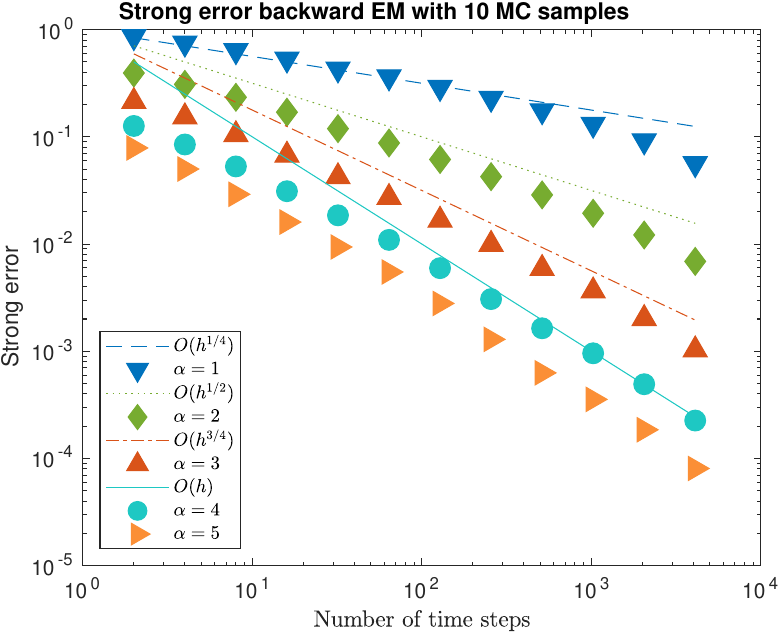} 
  	\caption{Backward Euler--Maruyama.}
   \label{fig:f12}
  \end{subfigure}
  \caption{Strong convergence error based on $10$~Monte Carlo samples.}
  \label{fig:MC}
\end{figure}

\appendix

\section{Properties of the solution}\label{app:prop_sol}

Let us consider the expectation of the solution. It holds that
\begin{equation*}
	\E[X(t)]
	= \E[X_0] + \int_0^t \Delta_{\mathbb{S}^2} \E[X(s)] \, \dd s
\end{equation*}
due to the linearity of the expectation and the mean zero property of the $Q$-Wiener process. Setting $u(t) = \E[X(t)]$ and $u_0 = \E[X_0]$, we obtain that the expectation of~$X$ is the solution to the (deterministic) PDE
\begin{equation*}
	\partial_t u = \Delta_{\mathbb{S}^2} u
\end{equation*}
with initial condition $u(0) = u_0$.

This PDE is solved by the variations of constants formula
\begin{equation*}
	\E[X(t)]
	= u(t)
	= \sum_{\ell =  0}^{\infty} \sum_{m = - \ell}^{\ell} e^{- \ell (\ell +1)t} u_{\ell,m}^0 Y_{\ell,m}
	= \sum_{\ell =  0}^{\infty} \sum_{m = - \ell}^{\ell} e^{- \ell (\ell +1)t} \E[X_{\ell,m}^0] Y_{\ell,m},
\end{equation*}
where $u_{\ell,m}^0 = \langle u_0, Y_{\ell,m}\rangle_{L^2(\mathbb{S}^2;\C)}$. 

Another interesting quantity of the solution is the second moment $\E[\|X(t)\|^2_{L^2(\IS^2)}]$. We observe first that
\begin{equation*}
	\E[\|X(t)\|^2_{L^2(\IS^2)}]
	= \E\Bigl[\Bigl\|
	\sum_{\ell =  0}^{\infty} \sum_{m = - \ell}^{\ell} 
	\Bigl(e^{- \ell (\ell +1)t} X_{\ell,m}^0 + \int_0^t e^{- \ell (\ell +1)(t-s)} \, \dd a_{\ell,m}(s)\Bigr) 
	Y_{\ell,m}\Bigr\|_{L^2(\IS^2)}^2\Bigr],
\end{equation*}
where the stochastic processes $a_{\ell,m}$ are given in~\eqref{eqn: KarhLo}. Due to the independence of the $Q$-Wiener process and the initial condition and the mean zero property of the It\^o integral, the two terms separate. While the first term satisfies
\begin{equation*}
	\E\Bigl[\Bigl\|
	\sum_{\ell =  0}^{\infty} \sum_{m = - \ell}^{\ell} e^{- \ell (\ell +1)t} X_{\ell,m}^0 Y_{\ell,m}\Bigr\|_{L^2(\IS^2)}^2\Bigr]
	= \sum_{\ell =  0}^{\infty} \sum_{m = - \ell}^{\ell} e^{- 2 \ell (\ell +1)t} \E[|X_{\ell,m}^0|^2] \| Y_{\ell,m}\|^2_{L^2(\mathbb{S}^2;\C)},
\end{equation*}
it remains to have a closer look at the stochastic convolution next. By the It\^o isometry and the scaling of the spherical harmonic functions, we obtain
\begin{align*}
	& \E\left[\Bigl\|
	\sum_{\ell =  0}^{\infty} \sum_{m = - \ell}^{\ell} 
	\int_0^t e^{- \ell (\ell +1)(t-s)} \, \dd a_{\ell,m}(s)
	Y_{\ell,m}\Bigr\|_{L^2(\IS^2)}^2\right]\\
	& = \E\biggl[\Bigl\| \sum_{\ell =  0}^{\infty} \Bigl( \sqrt{A_{\ell}} \int_0^t e^{- \ell (\ell +1)(t-s)} \, \dd \beta_{\ell,0}^1(s) Y_{\ell,0} \\
	& \hspace*{2em} + \sqrt{2 A_{\ell}} \sum_{m=1}^{\ell} \bigl(\int_0^t e^{- \ell (\ell +1)(t-s)} \, \dd \beta_{\ell,m}^1(s) \mathrm{Re} Y_{\ell,m} + \int_0^t e^{- \ell (\ell +1)(t-s)} \, \dd \beta_{\ell,m}^2(s) \mathrm{Im} Y_{\ell,m}\bigr)\Bigr) 	\Bigr\|_{L^2(\IS^2)}^2\biggr]\\
	& = \sum_{\ell =  0}^{\infty} \biggl( 
	A_{\ell} \int_0^t e^{- 2 \ell (\ell +1)(t-s)} \, \dd s \,
	\Bigl( \|Y_{\ell,0}\|_{L^2(\IS^2)}^2 + 2 \sum_{m=1}^{\ell}  (\|\Re Y_{\ell,m}\|_{L^2(\IS^2)}^2 + \|\Im Y_{\ell,m}\|_{L^2(\IS^2)}^2 ) \Bigr) \biggr)\\
	& = \sum_{\ell =  0}^{\infty}A_\ell (2\ell(\ell+1))^{-1} (1 - e^{- 2 \ell (\ell+1)t}) (1+ 2\ell).
\end{align*}
In conclusion the second moment of~$X(t)$ is given by
\begin{align*}
	\begin{split}
		& \E[\|X(t)\|^2_{L^2(\IS^2)}] \\
		& \quad = \sum_{\ell =  0}^{\infty} \Bigl( \sum_{m = - \ell}^{\ell} e^{- 2 \ell (\ell +1)t} \E[|X_{\ell,m}^0|^2] \| Y_{\ell,m}\|^2_{L^2(\mathbb{S}^2)} \Bigr)
		 +  A_\ell (1+ 2\ell) (2\ell(\ell+1))^{-1} (1 - e^{- 2 \ell (\ell+1)t}).
	\end{split}
\end{align*}

\section{Regularity of exponential functions and their approximation}\label{app:reg_exp_approx}

In this section we collect the proofs on the regularity of exponential functions and their approximation with a forward and backward Euler method from the propositions in Section~\ref{sec-Euler}.

\begin{proof}[Proof of Proposition~\ref{prop:exp_fEM}]
	Let us start to prove the first property~\ref{prop:exp_fEM_1}.
	By partial integration we obtain that
	\begin{equation*}
		| e^{-\ell(\ell+1)h} - (1-\ell(\ell+1)h) |
		= \left| \int_0^h \int_0^s (\ell(\ell+1))^2 e^{-\ell(\ell+1)r} \, \dd r \, \dd s \right|.
	\end{equation*}
	Since $x^\eta e^{-x} \le \tilde{C}_\eta$, we can bound the expression inside the integral by
	\begin{equation*}
		(\ell(\ell+1))^2 e^{-\ell(\ell+1)r}
		\le \tilde{C}_\mu (\ell(\ell+1))^{1+ \mu} r^{\mu-1},
	\end{equation*}
	which leads for any $\mu \in (0,1]$ to
	\begin{align*}
		\left| \int_0^h \int_0^s (\ell(\ell+1))^2 e^{-\ell(\ell+1)r} \, \dd r \, \dd s \right|
		& \le  \tilde{C}_\mu (\ell(\ell+1))^{1+ \mu} \mu^{-1} \int_0^h s^{\mu} \, \dd s \\
		& = \tilde{C}_\mu (\ell(\ell+1))^{1+ \mu} \mu^{-1} (1+ \mu)^{-1} h^{1+ \mu}
		= C_\mu (\ell(\ell+1))^{1+ \mu} h^{1+ \mu}.
	\end{align*}
	
	We continue with the proof of~\ref{prop:exp_fEM_3} and use $a^n - b^n = (a-b) \sum_{j = 0}^{n-1} a^{n-1-j}b^j$ to obtain that
	\begin{align*}
		& | e^{- \ell ( \ell +1) h\cdot k} - (1-\ell (\ell +1)h)^k |\\
		& \quad =  |e^{- \ell ( \ell +1) h} - (1-\ell (\ell +1)h)| 
		\cdot \Bigl| \sum_{j = 0}^{k-1} e^{- \ell ( \ell +1) h\cdot j} (1-\ell (\ell +1)h)^{k-1-j} \Bigr|.
	\end{align*}
	The first term is bounded by~\ref{prop:exp_fEM_1} and for the second, we observe that the Taylor expansion with remainder satisfies
	\begin{equation*}
		e^{-\ell(\ell+1)h} = 1 - \ell(\ell+1)h + \int_0^{\ell(\ell+1)h} (\ell(\ell+1)h - s) e^{-s} \, \dd s.
	\end{equation*}
	Since the integral is positive, we obtain
	\begin{equation*}
		1 - \ell(\ell+1)h \le e^{-\ell(\ell+1)h},
	\end{equation*}
	which yields
	\begin{align*}
		\Bigl| \sum_{j = 0}^{k-1} e^{- \ell ( \ell +1) h\cdot j} (1-\ell (\ell +1)h)^{k-1-j} \Bigr|
		\le k \, e^{- \ell ( \ell +1) h\cdot (k-1)}
	\end{align*}
	and implies the first inequality of the claim. The second follows by
	\begin{equation*}
		k \, e^{- \ell ( \ell +1) h\cdot (k-1)}
			\le e^{\ell(\ell+1)h} \, \tilde{C}_1 (\ell ( \ell +1) h)^{-1}
			\le e^1 \tilde{C}_1 (\ell ( \ell +1) h)^{-1},
	\end{equation*}
	 applying again that $x^\eta e^{-x} \le \tilde{C}_\eta$ and that $\ell(\ell+1)h \le 1$.
\end{proof}

\begin{proof}[Proof of Proposition~\ref{prop:exp_bEM}]	
	Similarly to the proof of Proposition~\ref{prop:exp_fEM}, we observe first by partial integration that
	\begin{equation*}
		| e^{-\ell(\ell+1)h} - (1+\ell(\ell+1)h)^{-1} |
		= \left| \frac{-(\ell(\ell+1))^2}{1 + \ell(\ell+1)h} \int_0^h \int_s^h e^{-\ell(\ell+1)r} \, \dd r \, \dd s \right|.
	\end{equation*}
	Using again that $x^\eta e^{-x} \le \tilde{C}_\eta$ to bound $(\ell(\ell+1))^{1-\mu} e^{-\ell(\ell+1)r} \le \tilde{C}_{1-\mu} r^{\mu-1}$, we compute the integrals to obtain
	\begin{equation*}
		\int_0^h \int_s^h r^{\mu-1} \, \dd r \, \dd s
		= (1+\mu)^{-1} h^{1+\mu}.
	\end{equation*}
	Putting all together yields
	\begin{align*}
		\left| \frac{-(\ell(\ell+1))^2}{1 + \ell(\ell+1)h} \int_0^h \int_s^h e^{-\ell(\ell+1)r} \, \dd r \, \dd s \right|
		& \le \tilde{C}_{1-\mu} \frac{(\ell(\ell+1))^{1+\mu}}{1 + \ell(\ell+1)h} (1+\mu)^{-1} h^{1+\mu}\\
		& = C_\mu (\ell(\ell+1))^{1+\mu} h^{1+\mu}
	\end{align*}
	for all $\mu \in (-1,1]$, which concludes the proof of~\ref{prop:exp_bEM_1}.
	
	Using again the same approach as in Proposition~\ref{prop:exp_fEM} with $a^n - b^n = (a-b) \sum_{j = 0}^{n-1} a^{n-1-j}b^j$ and bounding
	\begin{align*}
		\left| \sum_{j = 0}^{k-1} e^{- \ell ( \ell +1) h\cdot j} (1+\ell (\ell +1)h)^{-(k-1-j)} \right|
		\le \frac{e^{C_c}}{1+C_c} k \, e^{- \ell ( \ell +1) h\cdot (k-1)}
	\end{align*}
	in a similar way yields both inequalities in~\ref{prop:exp_bEM_3}. The only difference is that we apply $\ell(\ell+1)h \le C_c$ to obtain the bound
	\begin{equation*}
		(1+\ell(\ell+1)h)^{-1}
			= (1+\ell(\ell+1)h)^{-1} e^{\ell(\ell+1)h} e^{-\ell(\ell+1)h}
			\le e^{C_c} e^{-\ell(\ell+1)h}. \qedhere
	\end{equation*}
\end{proof}

\begin{proof}[Proof of Proposition~\ref{prop:reg_exp}]
To prove~\ref{prop:reg_exp_2}, we observe first that 
\begin{align*}
	\int_{t_{j-1}}^{t_j} (e^{ - \ell (\ell+1)(t_k-s)} - e^{ - \ell (\ell+1)(t_k-t_{j-1})})^2  \, \dd s
		& = \int_{t_{j-1}}^{t_j} e^{ - 2\ell (\ell+1)(t_k-s)} (1 - e^{-\ell(\ell+1)(s-t_{j-1})})^2 \, \dd s \\
		& \le h e^{ - 2\ell (\ell+1)(t_k-t_j)} (1 - e^{-\ell(\ell+1)h})^2.
\end{align*}
Therefore we can bound
\begin{align*}
	\Bigl|\sum_{j = 1}^k \int_{t_{j-1}}^{t_j} (e^{ - \ell (\ell+1)(t_k-s)} - e^{ - \ell (\ell+1)(t_k-t_{j-1})})^2  \, \dd s\Bigr|
	\le h (1 - e^{-\ell(\ell+1)h})^2 \sum_{j = 1}^k e^{ - 2\ell (\ell+1)(t_k-t_j)}.
\end{align*}
Since $e^{ - 2\ell (\ell+1)h} < 1$ and $(1-x)^{-1} = \sum_{j=0}^\infty x^j$ for $|x|<1$, the sum satisfies
\begin{align*}
	\sum_{j = 1}^k e^{ - 2\ell (\ell+1)(t_k-t_j)}
		& = \sum_{j = 0}^{k-1} e^{ - 2\ell (\ell+1)h \cdot j}
		\le (1 - e^{ - 2\ell (\ell+1)h})^{-1}
		= (1 - e^{ - \ell (\ell+1)h})^{-1} (1 + e^{ - \ell (\ell+1)h})^{-1}\\
		& \le (1 - e^{ - \ell (\ell+1)h})^{-1}.
\end{align*}
On the one hand side, for $\mu \in (1/2,1]$
\begin{align*}
	1 - e^{-\ell(\ell+1)h}
		& = (\ell(\ell+1))^{2\mu-1} \int_0^h (\ell(\ell+1))^{2-2\mu} e^{-\ell(\ell+1)r} \, \dd r\\
		& \le \tilde{C}_{2-2\mu}(\ell(\ell+1))^{2\mu-1} \Bigl| \int_0^h r^{2\mu -2} \, \dd r \Bigr|
		= \tilde{C}_{2-2\mu}(\ell(\ell+1))^{2\mu-1} (2\mu-1)^{-1} h^{2\mu-1}.
\end{align*}
For $\mu = 1/2$ the expression is bounded and for $\mu \in (0,1/2)$ on the other hand side
\begin{equation*}
	1 - e^{-\ell(\ell+1)h}
		\le 1
		= (\ell(\ell+1)h)^{1-2\mu}(\ell(\ell+1)h)^{2\mu - 1}
		\le C_c (\ell(\ell+1)h)^{2\mu - 1},
\end{equation*}
since $\ell(\ell+1)h \le C_c$. Therefore for $\mu \in (0,1]$, the expression satisfies
\begin{equation*}
	1 - e^{-\ell(\ell+1)h}
		\le C (\ell(\ell+1)h)^{2\mu - 1},
\end{equation*}
and putting all terms together yields the claim.
Similarly one proves~\ref{prop:reg_exp_3}.

We continue with the proof of~\ref{prop:reg_exp_4}. With the same steps as in the proof of~\ref{prop:reg_exp_2}, we arrive at
\begin{align*}
	& \Bigl|\sum_{j = 1}^k \int_{t_{j-1}}^{t_j} e^{ - 2\ell (\ell+1)(t_k-s)} - e^{ - 2 \ell (\ell+1)(t_k-t_{j-1})}  \, \dd s\Bigr| \\
	& \quad \le \sum_{j = 1}^k h e^{ - 2\ell (\ell+1)(t_k-t_j)} (1 - e^{-2\ell(\ell+1)h})
	\le h (1 - e^{-2\ell(\ell+1)h})^{-1} (1 - e^{-2\ell(\ell+1)h})\\
	& \quad = h (\ell(\ell+1)h)^{\mu -1} (\ell(\ell+1)h)^{1- \mu}
	\le C_c (\ell(\ell+1))^{\mu -1} h^\mu,
\end{align*}
which shows the claim. The proof of~\ref{prop:reg_exp_5} follows in the same way.
\end{proof}

\bibliographystyle{abbrv}
\bibliography{dasbib}

\end{document}